\documentclass[a4paper,12pt]{amsart}
\usepackage{amssymb,amsmath}
\usepackage{amsopn}
\usepackage{mathrsfs}
\usepackage{color}
\usepackage{wasysym}
\usepackage{vmargin}
\usepackage{vmargin}
\usepackage{color}
\usepackage{amscd}
\usepackage{verbatim}
\newtheorem{theorem}{Theorem}[section]
\newtheorem{lemma}[theorem]{Lemma}
\newtheorem{proposition}[theorem]{Proposition}
\newtheorem{corollary}[theorem]{Corollary}

\numberwithin{equation}{section}

\newcommand\RR{{{\mathbb R}}}

\newcommand\SSS{{\mathbb S}}
\newcommand{\rr}{\mathbb{R}}
\newcommand{\eps}{\varepsilon}
\newcommand{\nn}{\mathbb{N}}
\newcommand{\cc}{\mathbb{C}}
\def\un{{\mathrm{1~\hspace{-1.4ex}l}}}

\def\triple{|\hspace{-0.4mm}|\hspace{-0.4mm}|}
\def\p{\partial}

\def\io{{\infty}}

\def\re{\operatorname{Re}}
\def\im{\operatorname{Im}}

\def\N{\mathbb N}

\def\R{\mathbb R}
\def\C{\mathbb C}
\def\poscal#1#2{\langle#1,#2\rangle}

\def\norm#1{\Vert#1\Vert}
\def\val#1{\vert#1\vert}

\def\valjp#1{\langle#1\rangle}

\def\l2{L^2(\R^{n})}
\def\L2{L^2(\R^{2n})}

\def\hs{{\hskip15pt}}
\def\vs{\vskip.3cm}

\let \dis=\displaystyle

\let \dis=\displaystyle

\def\mat22#1#2#3#4{\begin{pmatrix}#1&#2\\ #3&#4\end{pmatrix}}

\def\wt#1{\widetilde{#1}}

\def\finp{\operatorname{fp}}

\begin{document}

\title[Analysis of the non-cutoff Kac collision operator]
{Spectral and phase space analysis of the linearized non-cutoff Kac collision operator}
\author{N. Lerner, Y. Morimoto, K. Pravda-Starov  \& C.-J. Xu\\{}}
\date{\today}
\address{\noindent \textsc{N. Lerner, Institut de Math\'ematiques de Jussieu,
Universit\'e Pierre et Marie Curie (Paris VI),
4 Place Jussieu,
75252 Paris cedex 05,
France}}
\email{lerner@math.jussieu.fr}
\address{\noindent \textsc{Y. Morimoto, Graduate School of Human and Environmental Studies,
Kyoto University, Kyoto 606-8501, Japan}}
\email{morimoto@math.h.kyoto-u.ac.jp }
\address{\noindent \textsc{K. Pravda-Starov,
Universit\'e de Cergy-Pontoise,
CNRS UMR 8088,
D\'epartement de Math\'ematiques,
95000 Cergy-Pontoise, France}}
\email{karel.pravda-starov@u-cergy.fr}
\address{\noindent \textsc{C.-J. Xu, School of Mathematics, Wuhan university 430072, Wuhan, P.R. China\\
 and  \\
 Universit\'e de Rouen, CNRS UMR 6085, D\'epartement de Math\'ematiques, 76801 Saint-Etienne du Rouvray, France}}
\email{Chao-Jiang.Xu@univ-rouen.fr}
\begin{abstract}
The non-cutoff Kac operator is a kinetic model for the non-cutoff radially symmetric Boltzmann operator. For Maxwellian molecules, the linearization of the non-cutoff Kac operator around a Maxwellian distribution is shown to be a function of the harmonic oscillator, to be diagonal in the Hermite basis and to be essentially a fractional power of the harmonic oscillator.  This linearized operator is a pseudodifferential operator, and we provide
 a complete asymptotic expansion for its symbol in a class enjoying a nice symbolic calculus.
Related results for the linearized non-cutoff radially symmetric Boltzmann operator are also proven.
\end{abstract}

\maketitle

\section{Introduction}
\subsection{The Boltzmann equation} 
The Boltzmann equation describes the behaviour of a dilute gas when the only interactions taken into account are binary collisions~\cite{17}. It reads as the equation
\begin{equation}\label{e1}
\begin{cases}
\partial_tf+v\cdot\nabla_{x}f=Q(f,f),\\
f|_{t=0}=f_0,
\end{cases}
\end{equation}
for the density distribution of the particles $f=f(t,x,v) \geq 0$  at time $t$, having position $x \in \rr^d$ and velocity $v \in \rr^d$. The Boltzmann equation derived in 1872 is one of the fundamental equations in mathematical physics and, in particular, a cornerstone of statistical physics.

The term appearing in the right-hand-side of this equation $Q(f,f)$ is the so-called Boltzmann collision operator associated to the Boltzmann bilinear operator
\begin{equation}\label{eq1}
Q(g, f)=\int_{\rr^d}\int_{\SSS^{d-1}}B(v-v_{*},\sigma) \big(g'_* f'-g_*f\big)d\sigma dv_*,
\end{equation}
with $d \geq 2$,
where we are using the standard shorthand $f'_*=f(t,x,v'_*)$, $f'=f(t,x,v')$, $f_*=f(t,x,v_*)$, $f=f(t,x,v)$. In this expression, $v, v_*$ and $v',v_*'$ are the velocities in $\rr^d$ of a pair of particles respectively before and after the collision. They are connected through the formulas
$$v'=\frac{v+v_*}{2}+\frac{|v-v_*|}{2}\sigma,\quad   v_*'=\frac{v+v_*}{2}-\frac{|v-v_*|}{2}\sigma,$$
where the parameter $\sigma\in\SSS^{d-1}$ belongs to the unit sphere. Those relations correspond physically to elastic collisions with the conservations of momentum and kinetic energy in the binary collisions 
$$\quad v+v_{\ast}=v'+v_{\ast}', \quad |v|^2+|v_{\ast}|^2=|v'|^2+|v_{\ast}'|^2,$$
where $|\cdot|$ is the Euclidean norm on $\rr^d$.

For monatomic gas, the cross section $B(v-v_*,\sigma)$ is a non-negative function which only depends on the relative velocity $|v-v_*|$ and on the deviation angle $\theta$ defined through the scalar product in $\rr^d$,
$$\cos \theta=k \cdot \sigma, \quad k=\frac{v-v_*}{|v-v_*|}.$$
Without loss of generality, we may assume that $B(v-v_*,\sigma)$ is supported on the set where
$$k \cdot \sigma \geq 0,$$
i.e. where $0 \leq \theta \leq \frac{\pi}{2}$. Otherwise, we can reduce to this situation with the customary symmetrization
$$\tilde{B}(v-v_{*},\sigma)=\big[B(v-v_{*},\sigma)+B(v-v_{*},-\sigma)\big] \un_{\{\sigma \cdot k \geq 0\}},$$
with $\un_A$ being the characteristic function of the set $A$, since the term $f'f_*'$ appearing in the Boltzmann operator $Q(f,f)$ is invariant under the mapping $\sigma \rightarrow -\sigma$.
More specifically, we consider cross sections of  the type
\begin{equation}\label{eq1.01}
B(v-v_*,\sigma)=\Phi(|v-v_*|)b\Big(\frac{v-v_*}{|v-v_*|} \cdot \sigma\Big), 
\end{equation}
with a kinetic factor
\begin{equation}\label{sa0}
\Phi(|v-v_*|)=|v-v_*|^{\gamma}, \quad  \gamma \in ]-d,+\infty[,
\end{equation}
and a factor related to the deviation angle with a singularity
\begin{equation}\label{sa1}
(\sin \theta)^{d-2}b(\cos \theta)  \substack{\\ \\ \approx \\ \theta \to 0_{+} }  \theta^{-1-2s},
\end{equation}
for\footnote{The notation $a\approx b$ means $a/b$
is bounded from above and below by fixed positive constants.} some  $0 < s <1$. Notice that this singularity is not integrable
$$\int_0^{\frac{\pi}{2}}(\sin \theta)^{d-2}b(\cos \theta)d\theta=+\infty.$$
This non-integrability plays a major r\^ole regarding the qualitative behaviour of the solutions of the Boltzmann equation and this feature is essential for the smoothing effect to be present. Indeed, as first observed by Desvillettes for the Kac equation~\cite{D95}, grazing collisions that account for the non-integrability of the angular factor near $\theta=0$
do induce smoothing effects for the solutions of the non-cutoff Kac equation, or more generally for the solutions of the non-cutoff Boltzmann equation. On the other hand, these solutions are at most as regular as the initial data, see e.g. \cite{36}, when the cross section is assumed to be integrable, or after removing the singularity by using a cutoff function (Grad's angular cutoff assumption).

The physical motivation for considering this specific structure of cross sections is derived from particles interacting according to a spherical intermolecular repulsive potential of the form 
$$\phi(\rho)=\frac{1}{\rho^{r}}, \quad r>1,$$ 
with $\rho$ being the distance between two interacting particles. In the physical 3-dimensional space~$\rr^3$, the cross section satisfies the above assumptions with 
$$s=\frac{1}{r} \in ]0, 1[, \quad \gamma=1-4s \in ]-3, 1[.$$
For further details on the physics background and the derivation of the Boltzmann equation, we refer the reader to the extensive expositions \cite{17,villani2}.

In the present work, we study the non-cutoff Kac collision operator. The Kac operator is a one-dimensional collision model for the radially symmetric Boltzmann operator defined as
\begin{equation}\label{holl2}
K(g,f)=\int_{\val \theta\le \frac{\pi}{4}}\beta(\theta)\left(\int_{\RR} (g'_* f'-g_*f )dv_*\right)d\theta,
\end{equation}
with $f_*'=f(t,x,v_*')$, $f'=f(t,x,v')$, $f_*=f(t,x,v_*)$, $f=f(t,x,v)$, where the relations between pre and post collisional velocities given by
\begin{equation}\label{prepostk}
v'=v\cos\theta - v_*\sin\theta, \quad v'_* =v\sin\theta +v_*\cos\theta, \quad v,v_* \in \rr,
\end{equation}
follow from the conservation of the kinetic energy in the binary collisions
$$v^2+v_{\ast}^2=v'^2+v_{\ast}'^2$$ 
and where the cross section is an even non-negative function satisfying  
\begin{equation}\label{holl3}
\beta \geq 0, \quad \beta\in L^1_{\textrm{loc}}(0,1), \quad \beta(-\theta)=\beta(\theta).
\end{equation}
As for the Boltzmann operator, the main assumption concerning the cross-section is the presence of a non-integrable singularity for grazing collisions
\begin{equation}\label{ah1}
\beta(\theta) \ \substack{ \\ \dis\approx \\ \theta \to 0} |\theta|^{-1-2s},
\end{equation}
with $0<s<1$. Details about the definition of the Kac operator as a finite part integral are recalled in Section~\ref{kacsection}. In particular, when acting on functions depending only on the velocity variable, the function $K(g,f)$ is shown to belong to the Schwartz space $\mathscr S(\R_v)$ when $g,f \in \mathscr{S}(\rr_v)$ (Lemma \ref{6.lem.defkac}). As pointed out in~\cite{D95}, the non-integrability feature (\ref{ah1}) accounts for the diffusive properties of the non-cutoff Kac equation. We aim in this work at displaying the exact diffusive structure of the non-cutoff Kac operator. More specifically, we shall be concerned with a close-to-equilibrium framework and provide a complete spectral and microlocal analysis of the linearization of the non-cutoff Kac operator around a normalized Maxwellian distribution.

\subsection{The linearized Boltzmann operator}
We begin by recalling some properties of the linearized Boltzmann operator. We consider the linearization of the Boltzmann equation 
$$f=\mu+\sqrt{\mu}g,$$
around the Maxwellian equilibrium distribution 
\begin{equation}\label{maxwe}
\mu(v)=(2\pi)^{-\frac{d}{2}}e^{-\frac{|v|^2}{2}}.
\end{equation}
Since $Q(\mu,\mu)=0$ by the conservation of the kinetic energy, the Boltzmann operator $Q(f,f)$ can be split into three terms
$$Q(\mu+\sqrt{\mu}g,\mu+\sqrt{\mu}g)=Q(\mu,\sqrt{\mu}g)+Q(\sqrt{\mu}g,\mu)+Q(\sqrt{\mu}g,\sqrt{\mu}g),$$
whose linearized part is
$Q(\mu,\sqrt{\mu}g)+Q(\sqrt{\mu}g,\mu).$
Setting
\begin{equation}\label{jan1}
\mathscr{L}g=\mathscr{L}_{1}g+\mathscr{L}_{2}g,
\end{equation}
with
\begin{equation}\label{jan2}
\mathscr{L}_{1}g=-\mu^{-1/2}Q(\mu,\mu^{1/2}g), \quad \mathscr{L}_{2}g=-\mu^{-1/2}Q(\mu^{1/2}g,\mu),
\end{equation}
the original Boltzmann equation (\ref{e1}) is reduced to the Cauchy problem for the fluctuation
\begin{equation}\label{boltz}
\begin{cases}
\partial_tg+v\cdot\nabla_{x}g+\mathscr{L}g=\mu^{-1/2}Q(\sqrt{\mu}g,\sqrt{\mu}g),\\
g|_{t=0}=g_0.
\end{cases}
\end{equation}
The Boltzmann operator is local in the time and position variables and from now on, we consider it as acting only in the velocity variable.
This linearized operator is known \cite{17} to be an unbounded symmetric operator on $L^2(\rr^d_{v})$ (acting in the velocity variable) such that its Dirichlet form satisfies
$$(\mathscr{L}g,g)_{L^2(\rr^d_{v})} \geq 0.$$
Setting
$$\mathbf{P}g=(a+b \cdot v+c|v|^2)\mu^{1/2},$$
with $a,c \in \rr$, $b \in \rr^d$, the $L^2$-orthogonal projection onto the space of collisional invariants
\begin{equation}\label{coli}
\mathcal{N}=\textrm{Span}\big\{\mu^{1/2},v_1 \mu^{1/2},...,v_d\mu^{1/2},|v|^2\mu^{1/2}\big\},
\end{equation}
we have
\begin{equation}\label{ker}
(\mathscr{L}g,g)_{L^2(\rr^d)}=0 \Leftrightarrow g=\mathbf{P}g.
\end{equation}
For Maxwellian molecules, i.e. when $\gamma=0$ in the kinetic factor (\ref{sa0}), the spectrum of the linearized Boltzmann operator is only composed by eigenvalues explicitly computed in \cite{wang}. See also \cite{bobylev,17,dolera}. Cercignani \cite{cercignani} about forty years ago noticed that the linearized Boltzmann operator with Maxwellian molecules behaves like a fractional diffusive operator. Over the time, this point of view transformed into the following widespread heuristic conjecture on the diffusive behavior of the Boltzmann operator as a flat fractional Laplacian \cite{al-1,amuxy-2,villani2}: 
$$f \mapsto Q(\mu,f) \sim -(-\Delta_v)^sf+ \textrm{ lower order terms},$$
with $0<s<1$ being the parameter appearing in the singularity assumption (\ref{sa1}). See \cite{lmp,MoXu,M-X2} for works related to this simplified model of the non-cutoff Boltzmann equation. Regarding the linearized non-cutoff Boltzmann operator for general molecules, sharp coercive estimates in the weighted isotropic Sobolev spaces $H^k_l(\rr^d)$ were proven in \cite{amuxy3,amuxy-4-1,gr-st,44,strain}:
\begin{equation}\label{sa2b}
 \left\|
(1 - {\bf P} ) g\right\|^2_{H^s_{\frac{\gamma}{2}}}+\left\|
(1 - {\bf P} ) g\right\|^2_{L^2_{s+\frac{\gamma}{2}}}
\lesssim (\mathscr L g,\, g)_{L^2(\rr^d)} \lesssim \left\|
(1 - {\bf P} )g\right\|^2_{H^s_{s+\frac{\gamma}{2}}},
\end{equation}
where
$$H^k_l(\rr^d) = \big\{ f\in \mathscr S ' (\RR^d ) : \ (1+|v|^2)^{\frac{l}{2}} f \in H^k (\RR^d ) \big\}, \quad k, l \in \RR.$$
As a byproduct of our analysis of the linearized non-cutoff Kac operator, we investigate in this work this heuristic conjecture in the particular case of the linearized non-cutoff Boltzmann operator with Maxwellian molecules acting on radially symmetric functions with respect to the velocity variable. This linearized non-cutoff radially symmetric Boltzmann operator will be shown to be a function of the harmonic oscillator  
\begin{equation}\label{harmo}
\mathcal{H}=-\Delta_v+\frac{|v|^2}{4}
\end{equation}
and to be equal to the fractional harmonic oscillator 
$$\Big(1-\Delta_v+\frac{|v|^2}{4}\Big)^s,$$ 
up to some lower order terms, where $0<s<1$ is the parameter appearing in the singularity assumption \eqref{sa1}. We shall also display the exact phase space structure of this operator which will be shown to be a pseudodifferential operator
$$\mathscr{L}f=l^w(v,D_v)f,$$
when acting on radially symmetric Schwartz functions $f  \in \mathscr{S}_r(\rr_v^d)$, whose symbol admits a complete asymptotic expansion 
\begin{equation}\label{asymp}
l(v,\xi) \sim 
c_0\Big(1+\vert\xi|^2+\frac{|v|^2}{4}\Big)^{s}-d_{0}
+\sum_{k=1}^{+\infty}c_k\Big(1+\vert\xi|^2+\frac{|v|^2}{4}\Big)^{s-k}, 
\end{equation}
with $c_0,d_0>0$,  $c_k \in \rr$ when $k\geq 1$.
This asymptotic expansion provides a complete description of the phase space structure of the linearized non-cutoff radially symmetric Boltzmann operator and allows to strengthen in the radially symmetric case with Maxwellian molecules the coercive estimate (\ref{sa2b}) as
\begin{equation}\label{holl1}
\|\mathcal{H}^{\frac{s}{2}}(1-{\bf P})f\|_{L^2}^2  \lesssim (\mathscr{L}f,f)_{L^2} \lesssim \|\mathcal{H}^{\frac{s}{2}}(1-{\bf P})f\|_{L^2}^2 , \quad f \in \mathscr{S}_{r}(\rr^d),
\end{equation}
where $\mathcal H$ is the harmonic oscillator. However, let us mention that the general (non radially symmetric) Boltzmann operator is a truly anisotropic operator. This accounts in general for the difference between the lower and upper bounds in the sharp estimate (\ref{sa2b}). In the recent works \cite{amuxy-4-1,gr-st,gr-st1}, sharp coercive estimates for the general linearized non-cutoff Boltzmann operator were proven. In \cite{amuxy-4-1}, these sharp coercive estimates established in the three-dimensional setting $d=3$ (Theorem~1.1 in \cite{amuxy-4-1}),
\begin{equation}\label{tripleest}
\triple (1-{\bf P})f\triple_{\gamma}^2  \lesssim (\mathscr{L}f,f)_{L^2} \lesssim  \triple (1-{\bf P})f\triple_{\gamma}^2, \quad f \in \mathscr{S}(\rr^3),
\end{equation}
involve the anisotropic norm
\begin{equation}\label{tripleest1}
\triple f \triple_{\gamma}^2=\int_{\rr^3_{v} \times \rr^3_{v_*} \times \SSS^2_{\sigma}}|v-v_*|^{\gamma}b(\cos \theta)\big(\mu_*(f'-f)^2+f_*^2(\sqrt{\mu'}-\sqrt{\mu})^2\big)dvdv_*d\sigma,
\end{equation}
whereas in \cite{gr-st,gr-st1}, coercive estimates involving the anisotropic norms
$$\|f\|_{N^{s,\gamma}}^2=\|f\|_{L_{\gamma+2s}^2}^2+\int_{\rr^{d}}\int_{\rr^{d}}\langle v \rangle^{\frac{\gamma+2s+1}{2}} \langle v' \rangle^{\frac{\gamma+2s+1}{2}}\frac{|f(v)-f(v')|^2}{d(v,v')^{d+2s}}\un_{d(v,v') \leq 1}dv dv',$$
where
$$d(v,v')=\sqrt{|v-v'|^2+\frac{1}{4}(|v|^2-|v'|^2)^2},$$
were derived and a model of a fractional geometric Laplacian with the geometry of a lifted paraboloid in $\rr^{d+1}$ was suggested for interpreting the anisotropic diffusive properties of the Boltzmann collision operator.

\section{Main results}

\subsection{Main results for the linearized non-cutoff Kac operator}\label{main}
We consider the non-cutoff Kac collision operator (\ref{holl2}) whose cross section satisfies to the assumptions (\ref{holl3}) and (\ref{ah1}). As before for the Boltzmann equation, we consider the fluctuation around the normalized Maxwellian distribution
$$\mu(v)=(2\pi)^{-\frac{1}{2}}e^{-\frac{v^2}{2}}, \quad v \in \rr,$$
by setting
$$f=\mu+\sqrt{\mu}h.$$
Since $K(\mu,\mu)=0$ by conservation of the kinetic energy,
we may write
$$
K(\mu+\sqrt{\mu}h,\mu+\sqrt{\mu}h)=K(\mu,\sqrt{\mu}h)
+K(\sqrt{\mu}h,\mu)+K(\sqrt{\mu}h,\sqrt{\mu}h)$$
and consider the linearized Kac operator
\begin{equation}\label{linear-K}
\mathcal{K}h=\mathcal{K}_1h+\mathcal{K}_2h,
\end{equation}
with
\begin{equation}\label{linear-K1}
\mathcal{K}_1h=-\mu^{-1/2}K(\mu,\mu^{1/2}h), \quad \mathcal{K}_2h=-\mu^{-1/2}K(\mu^{1/2}h,\mu).
\end{equation}
The first result gives an operator-theoretical formula expressing the first part of the linearized non-cutoff Kac operator as a function of the contraction semigroup generated by the one-dimensional harmonic oscillator 
\begin{equation}\label{harmone}
\mathcal{H}=-\Delta_v+\frac{v^2}{4}.
\end{equation}
We refer the reader to section \ref{6.sec.harmo} for a reminder on classical notations and formulas for the harmonic oscillator and the Hermite functions.

\bigskip

\begin{theorem}\label{th1.1} 
The first part of the linearized non-cutoff Kac operator defined by 
$$\mathcal{K}_1f=-\mu^{-1/2}K(\mu,\mu^{1/2}f),$$ 
is equal to
$$\mathcal K_{1}=\int^{\frac{\pi}{4}}_{-\frac{\pi}{4}}\beta(\theta)\Big(1-(\sec\theta)^{\frac{1}{2}}{\exp\bigl(-\mathcal{H}\ln(\sec\theta) \bigr)}\Big)d\theta,$$
where $\mathcal{H}$ is the one-dimensional harmonic oscillator (\ref{harmone}) so that 
\begin{equation}\label{new002}
\mathcal{K}_1=\sum_{k\ge 1}\Big(\int^{\frac{\pi}{4}}_{-\frac{\pi}{4}}\beta(\theta)\bigl(1-(\cos \theta)^k\bigr)d\theta\Big) \mathbb P_{k},
\end{equation}
where the projections $\mathbb P_{k}$ onto the Hermite basis are described in Section \ref{6.sec.harmo}.
\end{theorem}

\bigskip

\noindent 
Let us underline that in the integrals appearing in the formula (\ref{new002}) the $L^1$ singularity at 0 of the function $\beta$ is erased by the factor $(1-(\cos \theta)^k)$ which vanishes at the second order. The integrals in Theorem~\ref{th1.1} are therefore to be understood in the sense of Lemma \ref{new003}. This first result shows that $\mathcal{K}_1$ is an unbounded nonnegative operator on $L^2(\rr)$ which is diagonal in the Hermite basis. Furthermore, the more precise calculation \eqref{6.asyvp} shows that the domain of the operator $\mathcal K_{1}$ can be taken as
\begin{equation}\label{new005}
\mathcal D=\Big\{u\in L^2({\rr}),\quad \sum_{k\ge 0} k^{2s}\Vert{\mathbb P_{k}u}\Vert_{L^2}^2< +\infty\Big\}=\{u\in L^2({\rr}),\ \mathcal H^s u\in L^2(\rr)\}.
\end{equation}
The next theorem provides an operator-theoretical formula expressing the second part of the linearized non-cutoff Kac operator as a function of the spectral projections of the one-dimensional harmonic oscillator:

\bigskip

\begin{theorem}\label{th1.111}
The second part of the linearized non-cutoff Kac operator defined by 
$$\mathcal{K}_2f=-\mu^{-1/2}K(\mu^{1/2}f,\mu),$$
is equal to
$$\mathcal{K}_2=-\sum_{l=1}^{+\infty}\Big(\int^{\frac{\pi}{4}}_{-\frac{\pi}{4}}\beta(\theta)(\sin \theta)^{2l}d\theta\Big)\mathbb P_{2l}.$$
Furthermore, there exist some positive constants $c_1, c_2>0$ such that
\begin{equation}\label{new006}
0\le -\mathcal K_{2}\le c_1\exp-c_2 \mathcal H,
\end{equation}
where $\mathcal{H}$ is the one-dimensional harmonic oscillator (\ref{harmone}) and $\mathbb P_{k}$ are the spectral projections onto the Hermite basis described in Section \ref{6.sec.harmo}.
\end{theorem}

\bigskip

\noindent
Let us notice that in the integrals appearing in Theorem~\ref{th1.111} the $L^1$ singularity at~0 of the function $\beta$ is erased by the factor $(\sin \theta)^{2l}$ which vanishes at order $2l\ge 2$. The operator $\mathcal{K}_2$, as well as $\mathcal H^N\mathcal L_{2}$ for any $N\in \N$, is a trace class operator on $L^2(\rr)$. As the first part of the linearized non-cutoff Kac operator, the second part $\mathcal{K}_2$ is also diagonal in the Hermite basis. We therefore obtain the following spectral decomposition of the linearized non-cutoff Kac operator:

\bigskip

\begin{theorem}\label{th1.2}
The linearized non-cutoff Kac operator defined by 
$$\mathcal{K}f=-\mu^{-1/2}K(\mu,\mu^{1/2}f)-\mu^{-1/2}K(\mu^{1/2}f,\mu),$$
is a non-negative unbounded operator on $ L^2(\rr)$ with domain $\mathcal{D}$ defined in (\ref{new005}). It is diagonal in the Hermite basis
\begin{equation}\label{sal1}
\mathcal{K}=\sum_{k\ge 1}\lambda_{k}\mathbb P_{k},
\end{equation}
with a discrete spectrum only composed by the non-negative eigenvalues 
\begin{equation}
\lambda_{2k+1}=\int^{\frac{\pi}{4}}_{-\frac{\pi}{4}}\beta(\theta)
\left(1-(\cos\theta)^{2k+1}\right)d\theta \geq 0, \quad  k \geq 0,
\end{equation}
\begin{equation}\label{sal1.11}
\lambda_{2k}=\int^{\frac{\pi}{4}}_{-\frac{\pi}{4}}\beta(\theta)
\left(1-(\cos\theta)^{2k}-(\sin \theta)^{2k}\right)d\theta \geq 0, \quad  k \geq 1,
\end{equation}
satisfying to the asymptotic estimates
\begin{equation}\label{3.llkkb}
\lambda_{k} \approx k^s\quad \text{when $k\rightarrow+\io$}.
\end{equation}
\end{theorem}

\bigskip

\noindent
We notice that the lowest eigenvalue zero corresponds to the fact that the Maxwellian distribution $\mu$ is an equilibrium
$$\mathcal{K} \mu^{1/2}=-\mu^{-1/2}K(\mu,\mu)-\mu^{-1/2}K(\mu,\mu)=0,$$
by conservation of the kinetic energy. We shall now relate these operator-theoretical properties to the phase space structure of the linearized non-cutoff Kac operator. To that end, we define for any $m \in \rr$ the symbol classes $\mathbf{S}^m(\rr^{2d})$  as the set of smooth functions $a(v,\xi)$ from  $\R^d\times \R^d$ into $\C$ satisfying to the estimates
\begin{equation}\label{2.symcl}
\forall (\alpha,\beta) \in \nn^{2d}, \exists C_{\alpha\beta}>0, \forall (v,\xi) \in \rr^{2d},
|\partial_v^{\alpha}\partial_{\xi}^{\beta}a(v,\xi)| \leq C_{\alpha,\beta} \langle (v,\xi) \rangle^{2m-|\alpha|-|\beta|},
\end{equation}
with $\valjp{(v,\xi)}=\sqrt{1+\val v^2+\val \xi^2}$. We consider the Weyl quantization of symbols in the class $\mathbf{S}^m(\rr^{2d})$
\begin{equation}\label{2.weylq}
a^w(v,D_v)u=\frac{1}{(2\pi)^d}\int_{\rr^{2d}}e^{i(v-y) \cdot \xi}a\Big(\frac{v+y}{2},\xi\Big)u(y)dyd\xi.
\end{equation}
Some reminders about the Weyl quantization are recalled in Section~\ref{6.sec.susub}. 
We notice in particular that the Weyl symbol of the $d$-dimensional harmonic oscillator
$$|\xi|^2+\frac{|v|^2}{4} \in \mathbf S^1(\R^{2d}),$$
is a first order symbol in this symbolic calculus. The symbol class $\mathbf{S}^{-\infty}(\rr^{2d})$ denotes the class
$\cap_{m \in \rr}\mathbf{S}^m(\rr^{2d})$. We define for $m\ge 0$ the Sobolev space
{\small
\begin{equation}\label{2.sobol}
B^m(\rr^d)=\{u \in L^2(\rr^d),\ \mathcal H^m u \in L^2(\rr^d)\}\hskip-3pt
=\hskip-3pt\Big\{u \in L^2(\rr^d),\ \sum_{k\ge 1}k^{2m}\norm{\mathbb P_{k}u}_{L^2}^2<\hskip-1pt+\io\Big\}
\end{equation}
}\noindent
and $B^{-m}(\rr^d)$
as the dual space of $B^m(\rr^d)$. It follows from the general theory of Sobolev spaces attached to a pseudodifferential calculus (see e.g. Section 2.6 in \cite{birkhauser}) that
$$\forall m\in \R, \quad B^m(\R^d)=\{u\in\mathscr S'(\R^d), \forall a\in \mathbf S^m(\R^{2d}), a^w u\in L^2(\R^d)\}.$$
For definiteness, we shall now make the following choice for the cross section 
\begin{equation}\label{sing}
\beta(\theta)=\frac{|\cos\frac{\theta}{2}|}{|\sin\frac{\theta}{2}|^{1+2s}},\quad
\val \theta\leq\frac{\pi}{4}.
\end{equation}
With that choice, we get a more precise equivalent than in Theorem~\ref{th1.2},
\begin{equation}\label{czero}
\lambda_{k}\sim c_{0}k^s\quad \text{when $k\rightarrow+\io$ with $c_{0}=\frac{2^{1+s}}{s}\Gamma(1-s).$}
\end{equation}

\bigskip

\begin{theorem}\label{th1}
Under the assumption \eqref{sing}, the linearized non-cutoff Kac operator 
$$\mathcal K=l^w(v,D_v),$$
is a pseudodifferential operator whose Weyl symbol $l(v,\xi)$ is real-valued, belongs to the symbol class $\mathbf{S}^{s}(\rr^2)$ with the following asymptotic expansion:
there exists a sequence of real numbers $(c_k)_{k \geq 1}$ such that
$$\forall N \geq 1, \ l(v,\xi)\equiv c_0 \Big(1+\xi^2+\frac{v^2}{4}\Big)^s-d_{0}+\sum_{k=1}^Nc_k\Big(1+\xi^2+\frac{v^2}{4}\Big)^{s-k} \mod \mathbf{S}^{s-N-1}(\rr^2),$$
where the two positive constants $c_0,d_{0}>0$ are defined in \eqref{czero} and \eqref{dzero}.
\end{theorem}

\bigskip

\noindent
This result shows that the linearized non-cutoff Kac operator is a pseudodifferential operator whose principal symbol is the same as for the fractional harmonic oscillator
$$c_{0}\Big(1-\Delta_v+\frac{v^2}{4}\Big)^s.$$
According to standard results about the phase space structure of the powers of positive elliptic pseudodifferential operators (see e.g. Section~4.4 in~\cite{nier}), we notice that the linearized non-cutoff Kac operator is equal to the fractional harmonic oscillator
$$c_{0}\Big(1-\Delta_v+\frac{v^2}{4}\Big)^s,$$
up to a bounded operator on $L^2(\rr)$. Let us underline that the fractional power $0<s<1$ of the harmonic oscillator only relates to structure of the singularity (\ref{ah1}) whereas the different constants $d_{0},(c_k)_{k \geq 0}$ appearing in the asymptotic expansion
\begin{equation}\label{asymp}
l(v,\xi) \sim 
c_0\Big(1+\xi^2+\frac{v^2}{4}\Big)^{s}-d_{0}
+\sum_{k=1}^{+\infty}c_k\Big(1+\xi^2+\frac{v^2}{4}\Big)^{s-k},
\end{equation}
may be computed explicitly and depend directly on the exact expression chosen for the angular factor (\ref{sing}). This asymptotic expansion provides a complete description of the phase space structure of the linearized non-cutoff Kac operator. As we shall see in the proof of Theorem~\ref{th1}, the two parts $\mathcal{K}_1$ and $\mathcal{K}_2$ account very differently in the
way the linearized non-cutoff Kac operator acts. The first part $\mathcal{K}_1$ is a pseudodifferential operator whose Weyl symbol $l_1$ accounts for all the asymptotic expansion of the symbol $l$, 
$$l_1(v,\xi) \sim c_0\Big(1+\xi^2+\frac{v^2}{4}\Big)^{s}-d_{0}+\sum_{k=1}^{+\infty}c_k\Big(1+\xi^2+\frac{v^2}{4}\Big)^{s-k},$$
whereas the symbol of the operator $\mathcal{K}_2$ belongs to the symbol class $\mathbf{S^{-\infty}}(\rr^{2})$. This shows that $\mathcal{K}_2$ is a smoothing operator in any direction of the phase space
$$\|\langle v \rangle^{N_1}\mathcal{K}_2f\|_{H^{N_2}(\rr)} \lesssim \|f\|_{L^2(\rr)},$$
for all $N_1,N_2 \in \nn$, $f \in \mathscr{S}(\rr)$ and that $\mathcal K_{2}$ defines a compact operator on $L^2(\rr)$.

\subsection{Main results for the linearized non-cutoff radially symmetric Boltzmann operator}\label{resbolt}
We consider the linearized non-cutoff Boltzmann operator defined in (\ref{jan1}) with Maxwellian molecules  
$$\mathscr{L}f=-\mu^{-1/2}Q(\mu,\mu^{1/2}f)-\mu^{-1/2}Q(\mu^{1/2}f,\mu),$$
acting on the radially symmetric Schwartz space on $\rr^d$ (see Section \ref{6.sec.radia}) with $d \geq 2$,
{\small
\begin{equation}\label{2.radia}
\mathscr{S}_{r}(\rr^d)=\big\{f\in\mathscr{S}(\rr^d),\forall v\in\rr^d,\forall A\in O(d), f(v)=f(Av)\big\}=\big\{f(\val v)\big\}_{\substack{f \text{ even }\\ f \in \mathscr S(\R)}},
\end{equation}
}\noindent
where $O(d)$ stands for the orthogonal group of $\rr^d$. We recall that the case of Maxwellian molecules corresponds to the case when $\gamma=0$ in the kinetic factor (\ref{sa0}) and that the non-negative cross section $b(\cos \theta)$ is assumed to be supported where $\cos \theta \geq 0$ and to satisfy the assumption (\ref{sa1}). We define the following function
\begin{equation}\label{new001}
\beta(\theta)=|\SSS^{d-2}||\sin 2\theta|^{d-2}b(\cos 2\theta) \substack{\\ \\ \approx \\ \theta \to 0 } \val\theta^{-1-2s}.
\end{equation}
The first result gives an operator-theoretical formula expressing the first part of the linearized non-cutoff radially symmetric Boltzmann operator as a function of the contraction semigroup generated by the $d$-dimensional harmonic oscillator:

\bigskip

\begin{theorem}\label{th1.1bis}
When it acts on the function space $\mathscr{S}_{\textrm{r}}(\rr^d)$, the first part of the linearized non-cutoff Boltzmann operator with Maxwellian molecules defined by
$$\mathscr{L}_1f=-\mu^{-1/2}Q(\mu,\mu^{1/2}f),$$ 
is equal to
\begin{equation}\label{312n}
\mathcal{L}_1=\int^{\frac{\pi}{4}}_{-\frac{\pi}{4}}\beta(\theta)
\left(1-(\sec\theta)^{\frac{d}{2}}{\exp\big(-\mathcal{H}\ln(\sec\theta) \big)}\right)d\theta,
\end{equation}
where $\beta$ is the function defined in \eqref{new001} and $\mathcal H=-\Delta_v+\frac{|v|^2}{4}$ is the $d$-dimensional harmonic oscillator so that
\begin{equation}\label{new002c}
\mathcal{L}_1=\sum_{k\ge 1}\Big(\int^{\frac{\pi}{4}}_{-\frac{\pi}{4}}\beta(\theta)\bigl(1-(\cos \theta)^k\bigr)d\theta\Big) \mathbb P_{k},
\end{equation}
where the projections $\mathbb P_{k}$ onto the Hermite basis are described in Section \ref{6.sec.harmo}.
\end{theorem}

\bigskip

\noindent
As for the first part of the linearized non-cutoff Kac operator, the domain of the operator $\mathcal L_{1}$ can be taken as
\begin{equation}\label{new005c}
\mathcal D=\Big\{u\in L^2({\rr^d}), \sum_{k\ge 0} k^{2s}\Vert{\mathbb P_{k}u}\Vert_{L^2}^2< +\infty\Big\}=\{u\in L^2({\rr^d}), \mathcal H^s u\in L^2(\rr^d)\}.
\end{equation}
Similarly to the second part of the linearized non-cutoff Kac operator, the next
theorem provides an operator-theoretical formula expressing the second part of the linearized non-cutoff radially symmetric Boltzmann operator as function of the spectral projections of the harmonic oscillator.

\bigskip

\begin{theorem}\label{th1.111bis}
When it acts on the function space $\mathscr{S}_{\textrm{r}}(\rr^d)$, the second part of the linearized non-cutoff Boltzmann operator with Maxwellian molecules defined by
$$\mathscr{L}_2f=-\mu^{-1/2}Q(\mu^{1/2}f,\mu),$$
is equal to
\begin{equation}\label{2.regul}\mathcal{L}_2=-\sum_{l\ge 1}\Big(\int^{\frac{\pi}{4}}_{-\frac{\pi}{4}}\beta(\theta)(\sin \theta)^{2l}d\theta\Big)\mathbb P_{2l},
\end{equation}
where $\beta$ is the function defined in \eqref{new001} and $\mathbb P_{k}$ are the spectral projections onto the Hermite basis described in Section \ref{6.sec.harmo}.
Furthermore, there exist some positive constants $c_1, c_2>0$ such that
\begin{equation}\label{new006}
0\le -\mathcal L_{2}\le c_1\exp-c_2 \mathcal H,
\end{equation}
where $\mathcal{H}=-\Delta_v+\frac{|v|^2}{4}$ is the $d$-dimensional harmonic oscillator.
\end{theorem}

\bigskip

\noindent
Collecting the two previous results and using the fact that $\mathbb{P}_{2k+1}f=0$ when $k \geq 0$ and $f \in \mathscr{S}_r(\rr^d)$, 
we recover in the radially symmetric case the spectral diagonalization obtained in \cite{wang} for the linearized Boltzmann operator: 

\bigskip

\begin{corollary}\label{th1.111bisbisbis}
When it acts on the function space $\mathscr{S}_{r}(\rr^d)$, the linearized non-cutoff Boltzmann operator with Maxwellian molecules 
$$\mathscr{L}f=-\mu^{-1/2}Q(\mu,\mu^{1/2}f)-\mu^{-1/2}Q(\mu^{1/2}f,\mu),$$
is equal to
$$\mathcal{L}=\sum_{k\ge 1}\Big(\int^{\frac{\pi}{4}}_{-\frac{\pi}{4}}\beta(\theta)\bigl(1-(\sin \theta)^{2k}-(\cos\theta)^{2k}\bigr)d\theta\Big) \mathbb P_{2k},$$
where $\beta$ is the function defined in \eqref{new001} and $\mathbb P_{k}$ are the spectral projections onto the Hermite basis described in Section \ref{6.sec.harmo}. Furthermore, the estimates
\begin{equation}\label{appeig}
\int^{\frac{\pi}{4}}_{-\frac{\pi}{4}}\beta(\theta)\bigl(1-(\sin \theta)^{2k}-(\cos\theta)^{2k}\bigr)d\theta \approx k^s\quad \text{when $k\rightarrow+\io$},
\end{equation}
 are satisfied and imply the following coercive estimates 
\begin{equation}\label{appeig1}
 \|\mathcal{H}^{\frac{s}{2}}(1-\mathbf{P}) f\|_{L^2}^2 \lesssim (\mathscr{L}f,f)_{L^2}\lesssim \|\mathcal{H}^{\frac{s}{2}}(1-\mathbf{P}) f\|_{L^2}^2,
\end{equation}
for $f \in \mathscr{S}_r(\rr^d)$, where $\mathcal{H}=-\Delta_v+\frac{|v|^2}{4}$ is the $d$-dimensional harmonic oscillator.
\end{corollary}

\bigskip

\noindent
Let us mention that the results of Theorems~\ref{th1.1bis}, \ref{th1.111bis} and Corollary~\ref{th1.111bisbisbis} (except for \eqref{appeig} and \eqref{appeig1}) hold true as well for the cutoff case when $\beta$ is integrable. For definiteness, we shall now make the following choice for the cross section
\begin{equation}\label{new001'}
\beta(\theta)=|\SSS^{d-2}||\sin 2\theta|^{d-2}b(\cos 2\theta)=\frac{|\cos\frac{\theta}{2}|}{|\sin\frac{\theta}{2}|^{1+2s}}.
\end{equation}
With that choice, we get as before a more precise equivalent than in Corollary~\ref{th1.111bisbisbis}
$$\int^{\frac{\pi}{4}}_{-\frac{\pi}{4}}\beta(\theta)\bigl(1-(\sin \theta)^{2k}-(\cos\theta)^{2k}\bigr)d\theta \sim c_{0}(2k)^s,$$
when $k\rightarrow+\io$, where the positive constant $c_0>0$ is defined in (\ref{czero}).

\bigskip

\begin{theorem}\label{th1bis}
Under the assumption \eqref{new001'}, the linearized non-cutoff Boltzmann operator with Maxwellian molecules acting on the radially symmetric function space $\mathscr{S}_r(\rr^d)$ is equal to a pseudodifferential operator
$$\mathscr{L}f=l^w (v, D_v)f, \quad f \in \mathscr{S}_r(\rr^d),$$
whose Weyl symbol $l(v,\xi)$ is real-valued, belongs to the symbol class $\mathbf{S}^{s}(\rr^{2d})$ with the following asymptotic expansion: there exists a sequence of real numbers $(c_k)_{k \geq 1}$ such that $\forall N \geq 1,$
$$ \ l(v,\xi)\equiv
c_0 \Big(1+|\xi|^2+\frac{|v|^2}{4}\Big)^s-d_{0}+\sum_{k=1}^Nc_k\Big(1+|\xi|^2+\frac{|v|^2}{4}\Big)^{s-k} \mod \mathbf{S}^{s-N-1}(\rr^{2d}),$$
where $|\cdot|$ is the Euclidean norm and $c_0, d_{0}>0$ are the positive constants defined in \eqref{czero} and \eqref{dzero}.
\end{theorem}

\bigskip

\noindent
This result shows that when acting on the function space $\mathscr{S}_{r}(\rr^d)$, the linearized non-cutoff Boltzmann operator with Maxwellian molecules is a pseudodifferential operator whose principal symbol is the same as for the  fractional harmonic oscillator
$$c_{0}\Big(1-\Delta_v+\frac{|v|^2}{4}\Big)^s.$$
For Maxwellian molecules, this accounts for the exact diffusive structure of the linearized non-cutoff radially symmetric Boltzmann operator and shows that this operator is equal to the fractional harmonic oscillator
$$c_{0}\Big(1-\Delta_v^2+\frac{|v|^2}{4}\Big)^s,$$
up to a bounded operator on $L^2(\rr^d)$. Let us mention that the phase space structure of the linearized non-cutoff Boltzmann was first investigated in~\cite{pao1,pao2}, but these results were somehow controversial (see remarks in~\cite{cer,klaus}). In these works, the linearized non-cutoff Boltzmann with Maxwellian molecules and $s=1/4$ in the assumption (\ref{sa1}) was shown to be a pseudodifferential operator whose symbol in the standard quantization satisfies to the following estimates
$$\exists c_1,c_2>0, \  \textrm{Re }p(v,\xi) > c_1(|\xi|^2+|v|^2)^{\frac{1}{4}}-c_2,$$
$$|p(v,\xi)| \lesssim \langle v \rangle^{\frac{1}{2}}\langle \xi \rangle^{\frac{1}{2}}, \quad \forall \alpha, \beta \in \nn^3, \ |\alpha|+|\beta| \geq 1, \  |\partial_v^{\alpha}\partial_{\xi}^{\beta}p(v,\xi)| \lesssim \langle (v,\xi) \rangle^{\frac{1}{2}}.$$ 
{}From a microlocal view point, these estimates are of a limited interest since the above estimates only point out that the symbol $p$ belongs to a gainless symbol class without any asymptotic calculus. In the radially symmetric case, the situation is much more favorable since the Weyl symbol of the linearized non-cutoff Boltzmann operator with Maxwellian molecules belongs to $\mathbf{S}^{s}(\rr^{2d})$ which is a standard symbol class enjoying nice symbolic calculus (see Lemma~2.2.18 in~\cite{birkhauser}). Indeed, the function space $\mathbf{S}^{m}(\rr^{2d})$ which writes with H\"ormander's convention as
$$S\Big(\langle (v,\xi) \rangle^{2m},\frac{|dv|^2+|d\xi|^2}{\langle (v,\xi) \rangle^2} \Big),$$
is a symbol class with gain $\lambda=\langle (v,\xi) \rangle^2$ in the symbolic calculus
$$\underbrace{a_1}_{\in S^{m_1}} \sharp^w \underbrace{a_2}_{\in S^{m_2}} =\underbrace{a_1 a_2}_{\in S^{m_1+m_2}} +\frac{1}{2i}\underbrace{\{a_1,a_2\}}_{\in S^{m_1+m_2-1}}+...$$
As for the linearized non-cutoff Kac operator, the two operators $\mathcal{L}_1$ and $\mathcal{L}_2$ defined in Theorem~\ref{th1.1bis}, \ref{th1.111bis} account very differently in the
way the operator $l^w(v,D_v)$ acts on functions. The first part $\mathcal{L}_1$ is a pseudodifferential operator whose Weyl symbol accounts for all the asymptotic expansion of the symbol $l$, 
$$l_1(v,\xi) \sim c_0\Big(1+\vert\xi|^2+\frac{|v|^2}{4}\Big)^{s}-d_{0}+\sum_{k=1}^{+\infty}c_k\Big(1+\vert\xi|^2+\frac{|v|^2}{4}\Big)^{s-k},$$
whereas the symbol of the operator $\mathcal{L}_2$ belongs to the class $S^{-\infty}(\rr^{2d})$. This shows that $\mathcal{L}_2$ is a smoothing operator in any direction of the phase space
$$\|\langle v \rangle^{N_1}\mathcal{L}_2f\|_{H^{N_2}(\rr^d)} \lesssim \|f\|_{L^2(\rr^d)},$$
for all $N_1,N_2 \in \nn$, $f \in \mathscr{S}(\rr^d)$ and that $\mathcal L_{2}$ defines a compact operator on $L^2(\rr^d)$.

All the results about the linearized non-cutoff radially symmetric Boltzmann operator are a byproduct of the analysis of the linearized non-cutoff Kac operator. More specifically, the article is organized as follows. Section~\ref{proofkac} is devoted first to some reminders about the Mehler formula and to the proofs of the results for the linearized non-cutoff Kac operator. The link between the Kac operator and the radially symmetric Boltzmann operator together with the proofs of the results concerning the radially symmetric Boltzmann operator are given in Section~\ref{radially}. The appendix in Section~\ref{appendix} provides a useful lemma to handle singular kernels, some formulas for the collision operators and the statement of Bobylev formulas used in the previous sections.

\section{Proof of the results}\label{proofkac}

\subsection{The Mehler formula}\label{mehl}
We begin by recalling the Mehler formula which will play an important r\^ole in our analysis. The Mehler formula provides an explicit formula for the Weyl symbol of the semigroup generated by the harmonic oscillator
$$\mathcal{H}=-\Delta_v+\frac{\val v^2}{4}.$$
Let
$$ (v,\xi) \in \R^d\times\R^d \mapsto q(v,\xi)\in \C,$$
be a complex-valued quadratic form with a positive definite real part $\textrm{Re }q \gg0$. Associated to this quadratic symbol is the Hamilton map $F \in M_{2d}(\cc)$ uniquely defined by the identity
$$q\big{(}(v,\xi);(y,\eta) \big{)}=\sigma \big{(}(v,\xi),F(y,\eta) \big{)}, \quad (v,\xi) \in \rr^{2d},\  (y,\eta) \in \rr^{2d},$$
where $q(\cdot;\cdot)$ stands for the polarized form
associated to $q$ and $\sigma$ is the canonical
symplectic form on $\rr^{2d}$,
$$\sigma \big{(}(v,\xi),(y,\eta) \big{)}=\xi \cdot y-v\cdot \eta, \quad (v,\xi) \in \rr^{2d},\  (y,\eta) \in \rr^{2d}.$$
The differential operator defined by  the Weyl quantization of the quadratic symbol~$q$,
$$q^w(v,D_v)u(v)=\frac{1}{(2\pi)^d}\int_{\rr^{2d}}e^{i (v-y) \cdot \xi}q\Big(\frac{y+v}{2},\xi\Big)u(y)dyd\xi,$$
equipped with the domain
$$D(q)=\big\{u \in L^2(\rr^d) :  q^w(v,D_v) u \in L^2(\rr^d)\big\},$$
is maximally accretive
$$\textrm{Re}\big(q^w(v,D_v)u,u\big)_{L^2} \geq 0, \quad u \in D(q).$$
This operator generates a contraction semigroup $(e^{-tq^w})_{t \geq 0}$ whose Weyl symbol
$$e^{-tq^w}=p_t^w(v,D_v),$$
is given by the Mehler formula \cite{mehler} (Theorem~4.2),
$$p_t(X)=\frac{\exp(-\sigma\big(X,\tan(tF)X))}{\sqrt{\det(\cos tF)}} \in \mathscr{S}(\rr^{2d}), \quad X=(v,\xi) \in \rr^{2d},$$
for any $t>0$. The Weyl symbol of the one-dimensional harmonic oscillator $\mathcal{H}$ is
$$q(v,\xi)=\xi^2+\frac{v^2}{4}.$$
A direct computation shows that its Hamilton map
$$F=\begin{pmatrix}
0 & 1 \\
-\frac{1}{4} & 0
\end{pmatrix},$$
satisfies 
$$F^2=-\frac{1}{4}I_{2d},\quad \cos(tF)= \cosh\Big(\frac{t}{2}\Big)I_{2d}, \quad \sin(tF)=2\sinh\Big(\frac{t}{2}\Big)F,$$
where $I_{2d}$ stands for the identity matrix.
This implies that
$$\tan(tF)=2\tanh\Big(\frac{t}{2}\Big)F, \quad \sqrt{\det(\cos tF)}=\cosh\Big(\frac{t}{2}\Big).$$
Moreover, we have
{\small\begin{align*}
\sigma\big(X,\tan(tF)X)\big)=2\tanh\Big(\frac{t}{2}\Big)\sigma(X,FX)
= {2}\tanh\Big(\frac{t}{2}\Big)q(X)=2\tanh\Big(\frac{ t}{2}\Big)\Big(\xi^2+\frac{v^2}{4}\Big).
\end{align*}}\noindent
By tensorization, we deduce that the Weyl symbol of the semigroup 
$$\exp{-t\mathcal H}= p_{t}^w(v,D_{v}),$$
generated by the $d$-dimensional harmonic oscillator $\mathcal{H}=-\Delta_v+\frac{|v|^2}{4}$ is given by
\begin{equation}\label{mehlerformula}
p_t(v,\xi)=\frac{\exp\big[-2\tanh(\frac{ t}{2})(|\xi|^2+\frac{|v|^2}{4})\big]}{\cosh^d(\frac{ t}{2})},
\end{equation}
for any $t \geq 0$. According to \eqref{6.harmo}, we may write
$$p_{t}^w(v,D_{v})=\sum_{k\ge 0}e^{-t(k+\frac d2)}\mathbb P_{k}.$$
Setting $z=\tanh(\frac{t}{2})$, we have
$$\frac t2=\tanh^{-1}z=\frac12\ln\Big(\frac{1+z}{1-z}\Big),\quad \cosh\Big(\frac{t}{2}\Big)=\cosh(\tanh^{-1}z)=\frac{1}{\sqrt{1-z^2}}.$$
Following \cite{unter} (p. 204-205), we obtain the following formula as $L^2(\R^d)$-bounded operators
\begin{equation}\label{6.knb44}
\left[\exp-\Bigl(2z\Bigl(\val \xi^2+\frac{\val v^2}{4}\Bigr)\Bigr)\right]^w=\frac{1}{(1+z)^d}\sum_{k\ge 0}\Big(\frac{1-z}{1+z}\Big)^k\mathbb P_{k},
\end{equation}
holding for any $z\in \C,\ \re z \ge 0$, the latter condition ensuring that $\val{\frac{1-z}{1+z}}\le 1$.
We may rewrite \eqref{mehlerformula} as 
\begin{equation}\label{6.ml2cc}
\left[\exp-\Bigl(2z\Bigl(\val \xi^2+\frac{\val v^2}{4}\Bigr)\Bigr)\right]^w
=\frac{1}{(1-z^2)^{\frac{d}{2}}}\exp\Big(\mathcal H\ln\frac{1-z}{1+z}\Big),
\end{equation}
when $\ \val z<1,\ \re z\ge 0$.
Notice that the condition $\val z<1$ ensures that $\textrm{Re}(\frac{1-z}{1+z})>0$
and $\re z\ge 0$
that $\re(\ln\frac{1-z}{1+z})\le 0$.
On the other hand, the identity \eqref{6.knb44} provides for $z=1$,
\begin{equation}\label{proj0}
\mathbb P_{0}=\Big[2^de^{-2(\val \xi^2+\frac{\val v^2}{4})}\Big]^w.
\end{equation}
We can reformulate \eqref{6.ml2cc} as
\begin{equation}\label{6.ln115}
\exp -2\mathcal H \zeta=(1-\tanh^2 \zeta)^{\frac{d}{2}}\Big[e^{-2(\val \xi^2+\frac{\val v^2}4)\tanh \zeta}
\Big]^w,
\end{equation}
for any for $\zeta\in\C$ such that $\re \zeta\ge 0, \val{\im \zeta}<\frac{\pi}{2}$. For the present analysis of the linearized non-cutoff Kac operator, additional identities linked to the Mehler formula are needed. 
For $\theta\in (-\frac{\pi}{2},\frac{\pi}{2})$, we set
$$t=2\tanh^{-1}\Big(\tan^2\Big(\frac{\theta}{2}\Big)\Big).$$
By using that 
$$\tanh^{-1}\zeta=\frac12\ln\Big(\frac{1+\zeta}{1-\zeta}\Big),$$
we obtain that 
$$t=\ln\Big(\frac{1+\tan^2(\frac{\theta}{2})}{1-\tan^2(\frac{\theta}{2})}\Big)=\ln \Bigl(\frac{(\cos \frac{\theta}{2})^{-2}}{2-(\cos \frac{\theta}{2})^{-2}}\Bigr)=\ln \Bigl(\frac{1}{\cos \theta}\Bigr)$$
and
$$\cosh\Big(\frac{t}{2}\Big)=\frac12\bigl((\cos\theta)^{-\frac{1}{2}}+(\cos\theta)^{\frac{1}{2}}\bigr)=\frac{\cos^2(\frac{\theta}{2})}{(\cos\theta)^{\frac{1}{2}}}.$$
As a result, it follows that
\begin{align*}
\left[\frac{e^{-2\tan^2(\frac{\theta}{2})(\val \xi^2+\frac{|v|^2}{4})}}{\cos^{2d}(\frac{\theta}{2})}\right]^w
&=\left[\frac{e^{-2\tanh(\frac{t}{2})(\val \xi^2+\frac{|v|^2}{4})}}{\cosh^d(\frac{t}{2})\cos^{2d}(\frac{\theta}{2})}
\right]^w\frac{\cos^{2d}(\frac{\theta}{2})}{(\cos \theta)^{\frac{d}{2}}}\\
&=\left[\frac{e^{-2\tanh(\frac{t}{2})(\val \xi^2+\frac{|v|^2}{4})}}
{\cosh^d(\frac{t}{2})}\right]^w(\cos \theta)^{-\frac{d}{2}}.
\end{align*}
Then, we deduce from \eqref{mehlerformula} that for any $\val \theta<\frac{\pi}{2}$,
\begin{equation}\label{6.444mm}
(\cos \theta)^{-\frac{d}{2}}\exp-\Bigl(\mathcal H\ln\Big(\frac1{\cos\theta}\Big)\Bigr)
=\left[\frac{e^{-2\tan^2(\frac{\theta}{2})(\val \xi^2+\frac{|v|^2}{4})}}
{\cos^{2d}(\frac{\theta}{2})}
\right]^w.
\end{equation}
For $0<\theta\le \frac{\pi}{2}$, we have 
$$z=\frac{\cos^{2}\theta}{(1+\sin\theta)^2}, \quad  \frac{1-z}{1+z}=\sin \theta,\quad 1-z^{2}=\frac{4\sin\theta}{(1+\sin\theta)^2}.$$
The formula \eqref{6.ml2cc} provides that
$$(1+\sin\theta)^{-d}\Big[e^{-2\frac{\cos^{2}\theta}{(1+\sin\theta)^2}({\val \xi}^2+\frac{|v|^2}{4})}\Big]^{w}=\frac{1}{2^d({\sin\theta})^{\frac{d}{2}}}\exp{\mathcal H\ln(\sin\theta)}=\frac1{2^d}\sum_{k\ge 0}(\sin \theta)^{k}\mathbb P_{k},$$
for any $0<\theta\le \frac{\pi}{2}$. This formula extends by analytic continuation 
$$(1+\sin\theta)^{-d}\Big[e^{-2\frac{\cos^{2}\theta}{(1+\sin\theta)^2}({\val \xi}^2+\frac{|v|^2}{4})}\Big]^{w}=\frac1{2^d}\sum_{k\ge 0}(\sin \theta)^{k}\mathbb P_{k},$$
for any $\val\theta<\frac{\pi}{2}$.
We deduce from \eqref{proj0} that for any $\val\theta<\frac{\pi}{2}$,
\begin{equation}\label{fork22}
\frac{\Big[e^{-2\frac{\cos^{2}\theta}{(1+\sin\theta)^2}({\val \xi}^2+\frac{|v|^2}{4})}\Big]^{w}}{(1+\sin\theta)^{d}}+
\frac{\Big[e^{-2\frac{\cos^{2}\theta}{(1-\sin\theta)^2}({\val \xi}^2+\frac{|v|^2}{4})}\Big]^{w}}{(1-\sin\theta)^{d}}=
\frac{1}{2^{d-1}}
\sum_{l\ge 0}(\sin \theta)^{2l}\mathbb P_{2l}
\end{equation}
and 
\begin{multline}\label{fork2}
2^{d-1}\frac{\Big[e^{-2\frac{\cos^{2}\theta}{(1+\sin\theta)^2}({\val \xi}^2+\frac{|v|^2}{4})}\Big]^{w}}{(1+\sin\theta)^{d}}+2^{d-1}\frac{\Big[e^{-2\frac{\cos^{2}\theta}{(1-\sin\theta)^2}({\val \xi}^2+\frac{|v|^2}{4})}\Big]^{w}}{(1-\sin\theta)^{d}}\\
-2^{d}\Big[e^{-2({\val \xi}^2+\frac{|v|^2}{4})}\Big]^{w}=\sum_{l\ge 1}(\sin \theta)^{2l}\mathbb P_{2l}.
\end{multline}

\subsection{Study of the linearized operator $\mathcal{K}_1$}\label{l1}
We study the first part of the linearized non-cutoff Kac operator
$$\mathcal{K}_{1}u=-\mu^{-1/2}K(\mu,\mu^{1/2}u),$$
defined in (\ref{linear-K1}) as a pseudodifferential operator given by the Weyl quantization of a symbol $l_{1}$.

\bigskip

\begin{lemma}\label{l1.1cutoff}
The Weyl symbol of the operator $\mathcal{K}_{1}$ is equal to
\begin{equation}\label{514jh}
l_{1}(v,\xi)=\int_{|\theta| \leq \frac{\pi}{4}}\beta(\theta)\Big[1-\sec^{2}\Big(\frac\theta2\Big)\exp\Big(-2 \tan^2\Big(\frac{\theta}{2}\Big)\Big(\xi^2+\frac{v^2}{4}\Big)\Big)\Big]d\theta.
\end{equation}
Furthermore, the operator $\mathcal K_{1}$ is equal to 
\begin{equation}\label{44dfs}
\mathcal K_{1}=\int_{\val \theta\le \frac{\pi}{4}}
\beta(\theta)
\left[1-(\sec \theta)^{\frac{1}{2}}\exp-\big(\mathcal H\ln\bigl({\sec\theta}\bigr)\big)\right]d\theta,
\end{equation}
is diagonal in the Hermite basis
\begin{equation}\label{44dfo}
\mathcal K_{1}=\sum_{k\ge 1}\Big(\int_{\val \theta\le \frac{\pi}{4}}
\beta(\theta)\bigl(1-(\cos \theta)^k\bigr)d\theta\Big) \mathbb P_{k},
\end{equation}
where
\begin{equation}\label{UU}
\int_{\val \theta\le \frac{\pi}{4}}\beta(\theta)\bigl(1-(\cos \theta)^k\bigr)d\theta
\approx k^{s},
\end{equation}
when $k\rightarrow+\io$.
\end{lemma}

\bigskip

\noindent
Notice that the functions of $\theta$ inside the integrals factoring $\beta$ are even, vanish at 0 and are smooth on the compact interval of integration. Lemma~\ref{new003} may therefore be applied and the symbol $l_{1}$ is indeed given by a Lebesgue integral.

\bigskip

\begin{proof}
Let $u$ be in the Schwartz space $\mathscr{S}(\rr)$.
It follows from the Bobylev formula (Lemma~\ref{prop2}) and the Fourier inversion formula that
\begin{multline*}
-\mu^{-1/2}K(\mu,\mu^{1/2}u)(v)\\=
\frac{e^{\frac{v^2}{4}}}{(2 \pi)^{\frac{3}{4}}}
\iint_{\rr\times(-\frac\pi{4},\frac\pi{4})}
\beta(\theta) \left[\widehat \mu(0)\widehat{\mu^{1/2}u}(\eta)-\widehat{\mu}(\eta \sin{\theta})\widehat{\mu^{1/2}u}(\eta\cos{\theta})\right]e^{iv\eta}d\eta d\theta.
\end{multline*}
Recalling the formula written in the $d$-dimensional case for future reference
\begin{equation}\label{eq4}
\widehat{\big(e^{-\frac{\alpha}{2} \val v^2}\big)}(\xi)=\int_{\R^{d}}e^{-\frac{\alpha}{2} \val v^2}e^{-iv\cdot\xi}dv=\frac{(2\pi)^{\frac{d}{2}}}{\alpha^{\frac{d}{2}}}e^{-\frac{\val\xi^2}{2\alpha}},
\end{equation}
when $\alpha>0$, we notice that $\widehat{\mu}(\xi)=e^{-\frac{\val\xi^2}{2}}$.
It follows that
\begin{align*}
&-\mu^{-1/2}K(\mu,\mu^{1/2}u)(v)\\
&\hs =\frac{1}{2 \pi}\iint_{\rr\times(-\frac\pi{4},\frac\pi{4})}\beta(\theta)\left(\int_{\R}e^{\frac{v^2-y^2}{4}}\left[e^{-iy\eta}-e^{-\frac{\eta^2\sin^2 \theta}{2}}e^{-iy\eta\cos{\theta}}\right]e^{iv\eta}u(y)dy\right)d\eta d\theta\\
&\hs =\int_{\val \theta\le \frac{\pi}{4}}\beta(\theta)(\mathcal K_{1,\theta} u)(v) d\theta,
\end{align*}
where the distribution-kernel of the operator $\mathcal K_{1,\theta}$ is  given by
the oscillatory integral
\begin{align*}
\mathfrak{K}_{1,\theta}(v,y)=& \ \frac{1}{2 \pi}\int_{\rr}e^{\frac{v^2-y^2}{4}} \left[e^{-iy\eta}-e^{-\frac{\eta^2\sin^2 \theta}{2}}e^{-iy\eta\cos{\theta}}\right]e^{iv\eta}
d\eta\\
 =& \ \delta_{0}(v-y)-\frac{1}{2 \pi}e^{\frac{v^2-y^2}{4}} \int_{\rr}e^{-\frac{\eta^2\sin^2 \theta}{2}}e^{-iy\eta\cos{\theta}}e^{iv\eta}d\eta\\
 =& \ \delta_{0}(v-y)-\frac{e^{\frac{v^2-y^2}{4}}}{\sqrt{2\pi} \val{\sin\theta}}\exp{-\frac{(v-y\cos \theta)^2}{2\sin^2\theta}}.
\end{align*}
Since
$$\mathfrak{K}_{1,\theta}\Big(v-\frac y2,v+\frac y2\Big)=\delta_{0}(y)-\frac{e^{-\frac{vy}{2}}}{\sqrt{2\pi} \val{\sin\theta}}
\exp{-\frac{(v-\frac y2-(v+\frac y2)\cos \theta)^2}{2\sin^2\theta}},$$
we deduce from \eqref{4.lk555} that the Weyl symbol of the operator $\mathcal K_{1,\theta}$ denoted $l_{1,\theta}$ writes as 
$$l_{1,\theta}(v,\xi)=1-\ell_{1,\theta}(v,\xi),$$
where
$$\ell_{1,\theta}(v,\xi)=\int_{\rr} e^{iy\xi}\frac{1}{\sqrt{2\pi}\val{\sin\theta}}\exp{-\frac{(v-\frac y2-(v+\frac y2)\cos \theta)^2+vy\sin^2\theta}{2\sin^2\theta}}dy.
$$
The numerator of the fraction in the exponential is
$$\Big[v(1-\cos \theta)-\frac y2(1+\cos \theta)\Big]^2+vy \sin^2\theta=4\sin^4\Big(\frac{\theta}{2}\Big)v^2+\cos^4\Big(\frac{\theta}{2}\Big)y^2,$$
so that the quadratic form in the variables $v,y$ is positive definite for any $0<\val \theta\le \frac{\pi}{4}$.
This leads to
\begin{align*}
\ell_{1,\theta}(v,\xi)&=
\frac{1}{\sqrt{2\pi}\val{\sin\theta}}
e^{-\frac{\sin^4(\frac{\theta}{2})v^2}{2\sin^2(\frac{\theta}{2})\cos^2(\frac{\theta}{2})}}\int_{\rr} e^{iy\xi}\exp{-\frac{\cos^4(\frac{\theta}{2})y^2}{8\sin^2(\frac{\theta}{2})\cos^2(\frac{\theta}{2})}} dy\\
&=\frac{1}{\cos^2(\frac{\theta}{2})}\exp{-\Big(2\tan^2\Big(\frac{\theta}{2}\Big)\Big(\xi^2+\frac{v^2}4\Big)\Big)}
\end{align*}
and
$$l_{1,\theta}(v,\xi)=1-\frac{\exp{-(2\tan^2(\frac{\theta}{2})(\xi^2+\frac{v^2}4))}}{ \cos^2(\frac{\theta}{2})}.$$
This proves \eqref{514jh}. Applying \eqref{6.444mm} in the one-dimensional case provides the formula \eqref{44dfs}.
Furthermore, \eqref{44dfo} and \eqref{UU} follow from \eqref{6.harmo} and \eqref{6.asyvp}. This ends the proof of Lemma 
\ref{l1.1cutoff}.
\end{proof}

\bigskip

\begin{lemma}\label{l1.1cutoff'}
When the cross section is given by \eqref{sing}, the Weyl symbol $l_1$ of the first part of the linearized non-cutoff Kac operator $\mathcal{K}_{1}$ belongs to the symbol class $\mathbf{S}^s(\rr^2)$. Furthermore, we have the following asymptotic equivalent
\begin{equation}\label{44dfk}
\int_{\val \theta\le \frac{\pi}{4}}\beta(\theta)\bigl(1-(\cos \theta)^k\bigr)d\theta
\sim \frac{2^{1+s}}{s}\Gamma(1-s)k^s
\quad\text{ when \ $k\rightarrow+\io$}.
\end{equation}
\end{lemma}

\bigskip

\begin{proof}
The asymptotic equivalent \eqref{44dfk} follows from \eqref{6.asyvp}. Setting 
$$\lambda=1+\xi^2+\frac{v^2}{4},$$
we use the substitution rule with $\tau=\tan^2(\frac{\theta}{2})$ in the formula \eqref{514jh} to get that 
\begin{align*}
l_{1}(v,\xi)&=2\int_{0}^{\tan^2(\frac{\pi}{8})}\frac{\cos(\frac{\theta}{2})}{\sin^{1+2s}(\frac{\theta}{2})}(1-(1+\tau)e^{-2\tau (\lambda-1)})\frac{d\tau}{(1+\tau)\tan(\frac{\theta}{2})}\\
&=2\int_{0}^{\tan^2(\frac{\pi}{8})}\frac{\cos^{-2s}(\frac{\theta}{2})}{\tan^{2+2s}(\frac{\theta}{2})}(1-(1+\tau)e^{-2\tau (\lambda-1)})\frac{d\tau}{1+\tau}
\\
&=2\int_{0}^{\tan^2(\frac{\pi}{8})}\frac{(1+\tau)^s}{\tau^{1+s}}(1-(1+\tau)e^{-2\tau (\lambda-1)})\frac{d\tau}{1+\tau}\\
&=2\int_{0}^{\tan^2(\frac{\pi}{8})}\underbrace{\tau^{-1-s}}_{u'(\tau)}\underbrace{((1+\tau)^{s-1}-(1+\tau)^s e^{-2\tau (\lambda-1)})}_{v(\tau)}d\tau,
\end{align*}
since $d\tau=\tan(\frac{\theta}{2})(1+\tan^2(\frac{\theta}{2})) d\theta$.
By recalling that  $0<s<1$ and integrating by parts, we obtain that
\begin{multline*}
l_{1}(v,\xi)=\frac{2}{s \tan^{2s}(\frac{\pi}{8})}\Big[\Big(1+\tan^2\Big(\frac{\pi}{8}\Big)\Big)^s \overbrace{e^{-2(\lambda-1)\tan^2(\frac{\pi}{8})}}^{\in {\mathbf S}^{-\io}}-\Big(1+\tan^2\Big(\frac{\pi}{8}\Big)\Big)^{s-1}\Big]\\
+\frac2s\int_{0}^{\tan^2(\frac{\pi}{8})}\bigl((s-1)(1+\tau)^{s-2}-s(1+\tau)^{s-1}e^{-2\tau (\lambda-1)}
+2(1+\tau)^s (\lambda-1) e^{-2\tau (\lambda-1)}
\bigr) \frac{d\tau}{\tau^{s}}
\\=-\frac{2(1+\tan^2(\frac{\pi}{8}))^{s-1}}{s \tan^{2s}(\frac{\pi}{8})}+\wt{l_{1}}(v,\xi)+{\mathbf S}^{-\io},
\end{multline*}
where
\begin{align*}
&\wt{l_{1}}(v,\xi)=\frac2s\int_{0}^{\lambda\tan^2(\frac{\pi}{8})}\frac{\lambda^{s-1}}{\sigma^s}\Big[(s-1)\Big(1+\frac\sigma\lambda\Big)^{s-2}-s\Big(1+\frac\sigma\lambda\Big)^{s-1}e^{-2\sigma}e^{\frac{2\sigma}{\lambda}} \\&\hskip230pt+2\Big(1+\frac\sigma\lambda\Big)^s (\lambda-1) e^{-2\sigma}e^{\frac{2\sigma}{\lambda}}
\Big] d\sigma
\\  = & \ \frac{2(s-1)}{s}\lambda^{s-1}\int_{0}^{\lambda\tan^2(\frac{\pi}{8})}\Big(1+\frac\sigma\lambda\Big)^{s-2} \frac{d\sigma}{\sigma^{s}}
+\frac{4\lambda^{s-1}(\lambda-1)}{s} \int_{0}^{\lambda\tan^2(\frac{\pi}{8})}\Big(1+\frac{\sigma}{\lambda}\Big)^s e^{-2\sigma} e^{\frac{2\sigma}{\lambda}}\frac{d\sigma}{\sigma^s}
\\&\hskip55pt -{2\lambda^{s-1}}\int_{0}^{\lambda\tan^2(\frac{\pi}{8})}\Big(1+\frac{\sigma}{\lambda}\Big)^{s-1}e^{-2\sigma} e^{\frac{2\sigma}{\lambda}}\frac{d\sigma}{\sigma^{s}}.
\end{align*}
The sum of the first and last terms writes as
\begin{equation}\label{4.l11te}
l_{1,1}(v,\xi)=\lambda^{s-1}\int_{0}^{\lambda\tan^2(\frac{\pi}{8})}\Big[\frac{2(s-1)}{s}
-2\Big(1+\frac{\sigma}{\lambda}\Big)e^{-2\sigma}e^{\frac{2\sigma}{\lambda}}\Big]\Big(1+\frac\sigma\lambda\Big)^{s-2}\frac{d\sigma}{\sigma^{s}}.
\end{equation}
The main term is the second one
\begin{equation}\label{4.l12te}
l_{1,2}(v,\xi)=\frac{2^{1+s}}{s}\lambda^{s-1}(\lambda-1)\int_{0}^{2\lambda\tan^2(\frac{\pi}{8})}\Big(1+\frac{w}{2\lambda}\Big)^s e^{\frac{w}{\lambda}}
e^{-w} \frac{dw}{w^{s}}.
\end{equation}
We may write
\begin{equation}\label{4.kjn11}
l_{1}(v,\xi)= -\frac{2(1+\tan^2(\frac{\pi}{8}))^{s-1}}{s\tan^{2s}(\frac{\pi}{8})} +{l_{1,1}}(v,\xi)+{l_{1,2}}(v,\xi)+{\mathbf S}^{-\io}.
\end{equation}
By using that $\tan^2(\frac{\pi}{8})<1$ and that the function
$$z\mapsto \kappa(z)=(1+z)^s e^{2z}=\sum_{j\ge 0}a_{j}z^j,$$ 
is holomorphic on $\val z<1$, we obtain that 
\begin{align*}
l_{1,2}&=\frac{2^{1+s}\lambda^{s-1}(\lambda-1)}{s}\sum_{j\ge 0} \frac{a_{j}}{2^{j}\lambda^{j}}\int_{0}^{2\lambda\tan^2(\frac{\pi}{8})}w^{j-s}e^{-w}dw\\
&=\frac{2^{1+s}\lambda^{s-1}(\lambda-1)}{s}\sum_{0\le j\le N} \frac{a_{j}}{2^{j}\lambda^{j}}\int_{0}^{2\lambda\tan^2(\frac{\pi}{8})}w^{j-s}e^{-w}dw\\
&+\frac{2^{1+s}\lambda^{s-1}(\lambda-1)}{s(2\lambda)^{N+1}}\int_{\rho=0}^1\int_{w=0}^{2\lambda\tan^2(\frac{\pi}{8})}\frac{(1-\rho)^N}{N!}
\kappa^{(N+1)}\Big(\frac{\rho w}{2\lambda}\Big)w^{N+1-s}e^{-w} d\rho dw\\
&=\frac{2^{1+s}\lambda^{s-1}(\lambda-1)}{s}\sum_{0\le j\le N} \frac{a_{j}}{2^{j}\lambda^{j}}\Gamma(1+j-s)\\
&-\underbrace{\frac{2^{1+s}\lambda^{s-1}(\lambda-1)}{s}\sum_{0\le j\le N} \frac{a_{j}}{2^{j}\lambda^{j}}\int_{2\lambda\tan^2(\frac{\pi}{8})}^{+\io}
w^{j-s}e^{-w}dw}_{\in \mathbf{S}^{-\io}}\\
&+\frac{2^{1+s}\lambda^{s-1}(\lambda-1)}{s(2\lambda)^{N+1}}\int_{\rho=0}^1\int_{w=0}^{2\lambda\tan^2(\frac{\pi}{8})}\frac{(1-\rho)^N}{N!}\kappa^{(N+1)}\Big(\frac{\rho w}{2\lambda}\Big)w^{N+1-s}e^{-w} d\rho dw.
\end{align*}
To prove that the last line belongs to ${\mathbf S}^{s-N-1}$, it is sufficient to prove that
$$\omega_{0}(v,\xi)=\int_{\rho=0}^1\int_{w=0}^{2\lambda\tan^2(\frac{\pi}{8})}\frac{(1-\rho)^N}{N!}\kappa^{(N+1)}\Big(\frac{\rho w}{2\lambda}\Big)
w^{N+1-s}e^{-w} d\rho dw\in {\mathbf S}^0.$$
To that end, we first notice that the function $\omega_{0}$ is bounded
\begin{multline*}
\val{\omega_{0}(v,\xi)}\le \int_{\rho=0}^1\int_{w=0}^{2\lambda\tan^2(\frac{\pi}{8})}\frac{(1-\rho)^N}{N!}\norm{\kappa^{(N+1)}}_{L^\io(\val z\le\tan^2(\frac{\pi}{8}))}w^{N+1-s}e^{-w}d\rho dw \\ \leq \frac{\Gamma(N+2-s)}{(N+1)!}\norm{\kappa^{(N+1)}}_{L^\io(\val z\le\tan^2(\frac{\pi}{8}))}.
\end{multline*}
Writing $\omega_{0}(v,\xi)=\Omega_{0}(\lambda(v,\xi))$, we have
\begin{multline*}
\frac{d\Omega_{0}}{d\lambda}=\Big(2\lambda\tan^2\Big(\frac\pi8\Big)\Big)^{N+2-s}\frac{e^{-2\lambda\tan^2(\frac\pi8)}}{\lambda}
\int_{0}^1\frac{(1-\rho)^N}{N!}\kappa^{(N+1)}\Big(\rho\tan^2\Big(\frac\pi8\Big)\Big)d\rho
\\-\frac{1}{\lambda}\int_{\rho=0}^1\int_{w=0}^{2\lambda\tan^2(\frac\pi8)}
\frac{(1-\rho)^N}{N!}\kappa^{(N+2)}\Big(\frac{\rho w}{2\lambda}\Big)\frac{\rho w}{2\lambda}w^{N+1-s}e^{-w}d\rho dw,
\end{multline*}
where the first line belongs to ${\mathbf S}^{-\io}$ and by following the exact same reasoning as for bounding the function $\omega_{0}$, we notice that the second line is bounded above in modulus by a constant times $\lambda^{-1}$. It follows that 
$$\val{\nabla_{v,\xi}\omega_{0}}\lesssim  \lambda^{-1}\val{\nabla_{v,\xi} \lambda}\lesssim \lambda^{-1/2}.$$
The higher-order derivatives may be handled in the very same way. It follows that for any $N \in \nn$,
\begin{equation}\label{4.14kkb}
l_{1,2}\equiv\sum_{0\le j\le N}\lambda^{s-j-1}(\lambda-1)\frac{2^{1+s-j}}{s}\Gamma(1+j-s)a_{j}\mod {\mathbf S}^{s-N-1}.
\end{equation}
The study of the term $l_{1,1}$ is very similar. We may write
\begin{multline}\label{4.nb556}
l_{1,1}(v,\xi)=2^{s}\lambda^{s-1}\int_{0}^{2\lambda\tan^2(\frac{\pi}{8})}\Big[\frac{s-1}{s}-\Big(1+\frac{w}{2\lambda}\Big)e^{-w}e^{\frac{w}{\lambda}}\Big]\Big(1+\frac w{2\lambda}\Big)^{s-2}\frac{dw}{w^{s}}\\
=2^{s}\frac{s-1}{s}\lambda^{s-1}\int_{0}^{2\lambda\tan^2(\frac{\pi}{8})}\Big(1+\frac w{2\lambda}\Big)^{s-2}\frac{dw}{w^{s}}
-2^{s}\lambda^{s-1}\int_{0}^{2\lambda\tan^2(\frac{\pi}{8})}\Big(1+\frac w{2\lambda}\Big)^{s-1}e^{\frac{w}{\lambda}}e^{-w}\frac{dw}{w^{s}},
\end{multline}
so that the last term is almost identical to the symbol $l_{1,2}$ (with leading term $\lambda ^{s-1}$) and the first integral in the last line is equal to the negative constant
$$-\frac{2(1-s)}{s}\int_{0}^{3-2^{3/2}}(1+t)^{s-2}\frac{dt}{t^{s}}.$$
We deduce from \eqref{4.kjn11} and \eqref{4.14kkb} that
\begin{equation}\label{4.asymp}
l_{1}\equiv \frac{2^{1+s}}{s}\Gamma(1-s)\lambda^s-d_{0}+\sum_{1\le j\le N}c_{j} \lambda^{s-j}\mod {\mathbf S}^{s-N-1},
\end{equation}
where
\begin{equation*}
d_{0}= \frac{2(1+\tan^2(\frac{\pi}{8}))^{s-1}}{s\tan^{2s}(\frac{\pi}{8})}
+\frac{2(1-s)}{s}\int_{0}^{3-2^{3/2}}(1+t)^{s-2}\frac{dt}{t^{s}}.
\end{equation*}
An easy calculation\footnote{Use the change of variable $t=\tan^{2}\theta$ in the integral term.} shows that
\begin{equation}\label{dzero}
d_{0}=\frac{2}{s\sin^{2s}(\frac{\pi}{8})}=\frac{2^{1+s}(2+\sqrt{2})^{s}}{s}.
\end{equation}
The formula \eqref{4.asymp} yields a full asymptotic expansion for $l_{1}$ as a symbol belonging to the class $\mathbf S^s(\R^2)$. This ends the proof of Lemma~\ref{l1.1cutoff'}.
\end{proof}

\subsection{Study of the linearized operator $\mathcal{K}_2$}\label{l2}
We consider the operator
\begin{align*}
\mathcal{K}_{2}u= & \ -\mu^{-1/2}K(\mu^{1/2}u,\mu)\\
= & \ -\mu^{-1/2}\int_{ |\theta| \leq \frac{\pi}{4}}\beta(\theta)
 \left(\int_{\R}\big((\mu^{1/2}\breve u)'_* \mu'-(\mu^{1/2}\breve u)_*\mu\big)dv_*\right)d\theta ,
\end{align*}
using the notation \eqref{6.evenp} and the expression \eqref{6.kac02}.

\bigskip

\begin{lemma}\label{prop1.2}
The Weyl symbol of the second part of the linearized non-cutoff Kac operator $\mathcal{K}_{2}$ is equal to
\begin{multline}\label{4.oper2}
l_{2}(v,\xi)=
\int_{|\theta| \leq \frac{\pi}{4}}
\beta(\theta)\Biggr[2 e^{-2(\xi^2+\frac{v^2}{4})}-
\frac{\exp\Big(-\frac{2\cos^2 \theta(\xi^2+\frac{v^2}{4})}{(1+\sin{\theta})^2}\Big)}{1+\sin \theta}
\\-\frac{\exp\Big(-\frac{2\cos^2\theta(\xi^2+\frac{v^2}{4})}{(1-\sin{\theta})^2}\Big)}{1-\sin \theta}
\Biggr]d\theta
\end{multline}
and satisfies
$$\forall (\alpha,\beta) \in \nn^2, \exists C_{\alpha,\beta}>0, \forall (v,\xi)\in \rr^2, \ |\partial_v^{\alpha}\partial_{\xi}^{\beta}l_{2}(v,\xi)| \leq C_{\alpha,\beta}e^{-\frac{1}{3}(\xi^2+\frac{v^2}{4})},$$
implying in particular that $l_{2} \in \mathbf{S^{-\infty}}(\rr^2)$.
\end{lemma}

\bigskip

\noindent
The formula \eqref{4.oper2} makes sense as an ordinary integral according to Lemma~\ref{new003}.

\bigskip

\begin{proof}
As in the previous section, we may use the Bobylev formula (Lem\-ma \ref{prop2}) to write
\begin{multline*}(\mathcal{K}_{2}u)(v)=
\\
\frac{e^{\frac{v^2}{4}}}{(2 \pi)^{\frac{3}{4}}}
\int_{|\theta| \leq \frac{\pi}{4}}\beta(\theta)
\left(\int_{\rr}
 \left[(\widehat{\mu^{1/2}\breve u})(0)\widehat{\mu}(\eta)
 -(\widehat{\mu^{1/2}\breve u})(\eta \sin{\theta})\widehat{\mu}(\eta\cos{\theta})\right]e^{iv\eta}d\eta\right) d\theta.
 \end{multline*}
It follows that
\begin{align*}
(\mathcal K_{2} u)(v)&=
\iint_{\rr\times(-\frac\pi{4},\frac\pi{4})}
\beta(\theta)\left(\frac{1}{2 \pi}\int_{\rr}
e^{\frac{v^2-y^2}{4}}
 \left[e^{-\frac{\eta^2}{2}}-e^{-\frac{\eta^2\cos^2 \theta}{2}}e^{-iy\eta\sin{\theta}}\right]e^{iv\eta}\breve u(y) dy\right)d\eta d\theta
\\
&=\int_{\val \theta\le \frac{\pi}{4}}\beta(\theta)(\mathcal K_{2,\theta} u)(v) d\theta,
 \end{align*}
where the distribution-kernel of the operator $\mathcal K_{2,\theta}$ is  given by
(see subsection \ref{6.sec.susub})
\begin{equation}\label{4.cx229}
\frac12\left(\mathfrak{K}_{2,\theta}(v,y)+\mathfrak{K}_{2,\theta}(v,-y)\right),
\end{equation}
whereas the oscillatory integral $\mathfrak{K}_{2,\theta}$ is
$$
\mathfrak{K}_{2,\theta}(v,y)=\frac{e^{\frac{v^2-y^2}{4}}}{2 \pi}
\int_{\rr}
 \left[e^{-\frac{\eta^2}{2}}-e^{-\frac{\eta^2\cos^2 \theta}{2}}e^{-iy\eta\sin{\theta}}\right]e^{iv\eta}
 d\eta.
 $$
By using \eqref{eq4}, we find that
 \begin{align*}
 \mathfrak{K}_{2,\theta}(v,y)&= \frac{e^{\frac{v^2-y^2}{4}}}{2 \pi}\Big[\sqrt{2\pi} e^{-\frac{v^2}{2}}-\frac{\sqrt{2\pi} }{\cos \theta}\exp-\frac{(v-y\sin \theta)^2}{2\cos^2 \theta}\Big] \\&= \frac{1}{\sqrt{2\pi}}\Big[e^{-\frac{y^2+v^2}{4}}-\frac{1}{\cos \theta}\exp-\Big(\frac{(v-y\sin \theta)^2}{2\cos^2 \theta}+\frac{y^2-v^2}{4}\Big)\Big].
\end{align*}
We obtain that
$$\sqrt{2\pi}\mathfrak{K}_{2,\theta}\Big(v-\frac y2,v+\frac y2\Big)=e^{-\frac{y^2+4v^2}{8}}-\frac{1}{\cos \theta}\exp-\Big(\frac{4(1-\sin \theta)^2v^2+(1+\sin \theta)^2y^2}{8\cos^2 \theta}\Big)$$
and
$$\sqrt{2\pi}\mathfrak{K}_{2,\theta}\Big(v-\frac y2,-v-\frac y2\Big)=e^{-\frac{y^2+4v^2}{8}} -\frac{1}{\cos \theta}\exp-\Big(\frac{4(1+\sin \theta)^2v^2+(1-\sin \theta)^2y^2}{8\cos^2 \theta}\Big).$$
Setting
$$l_{2,+,\theta}(v,\xi)=\int_{\rr}  \mathfrak{K}_{2,\theta}\Big(v-\frac y2,v+\frac y2\Big)e^{iy\xi} dy,$$
we deduce from \eqref{eq4} that
$$l_{2,+,\theta}(v,\xi)=2e^{-2(\xi^2+\frac{v^2}4)}-\frac{2}{1+\sin \theta}\exp-\Big(\frac{(1-\sin\theta)^2v^2}{2\cos^2\theta}+\frac{2(1-\sin\theta)\xi^2}{1+\sin\theta}\Big),$$
so that
$$l_{2,+,\theta}(v,\xi)=2e^{-2(\xi^2+\frac{v^2}4)}-\frac{2}{1+\sin \theta}\exp-\Big(\frac{2(1-\sin\theta)}{1+\sin\theta}\Big(\xi^2+\frac{v^2}4\Big)\Big).$$
It follows from \eqref{4.cx229} that the Weyl symbol $l_{2,\theta}$ of the operator $\mathcal K_{2,\theta}$ satisfies
\begin{align*}
2l_{2,\theta}(v,\xi)&=\int_{\rr}  \mathfrak{K}_{2,\theta}\Big(v-\frac y2,v+\frac y2\Big)e^{iy\xi} dy+\int_{\rr}  \mathfrak{K}_{2,\theta}\Big(v-\frac y2,-v-\frac y2\Big)e^{iy\xi} dy\\ &=l_{2,+,\theta}(v,\xi)+l_{2,+,-\theta}(v,\xi).
\end{align*}
We obtain that
\begin{multline*}
l_{2,\theta}(v,\xi)=2e^{-2(\xi^2+\frac{v^2}4)}-\frac{1}{1+\sin \theta}\exp-\Big(\frac{2(1-\sin\theta)}{1+\sin\theta}\Big(\xi^2+\frac{v^2}4\Big)\Big)\\
-\frac{1}{1-\sin \theta}\exp-\Big(\frac{2(1+\sin\theta)}{1-\sin\theta}\Big(\xi^2+\frac{v^2}4\Big)\Big).
\end{multline*}
We notice that the latter is an even, smooth function of the variable $\theta$ which vanishes at zero. By using Lemma~\ref{new003}, we obtain that  the Weyl symbol of the operator $\mathcal K_{2}$ is given by \eqref{4.oper2}. Furthermore, note that the function 
$$\theta \in \Big[-\frac{\pi}{4},\frac{\pi}{4}\Big]\mapsto \frac{2(1-\sin \theta)}{1+\sin\theta},$$ 
is valued in $[6-4\sqrt 2,6+4\sqrt 2]$.
Setting
$$\psi(\theta,v,\xi)=\frac{1}{1+\sin \theta}\exp-\Big(\frac{2(1-\sin\theta)}{1+\sin\theta}\Big(\xi^2+\frac{v^2}4\Big)\Big),$$
we easily check by induction on $\val \alpha+\val \beta$ that
$$\p_{\theta}^2\p_{v}^\alpha\p_{\xi}^\beta\psi=P_{\alpha,\beta}(v,\xi,1+\sin \theta)\exp-\Big(\frac{2(1-\sin\theta)}{1+\sin\theta}\Big(\xi^2+\frac{v^2}4\Big)\Big),$$
where $P_{\alpha,\beta}$ is a polynomial of degree $\val \alpha+\val \beta+4$ in the variables $(v,\xi)$ whose coefficients are rational fractions in the variable $1+\sin \theta$.
Since $6-4\sqrt 2>1/3$, the estimates following from   Lemma \ref{new003}
give the last statement of Lemma
\ref{prop1.2} whose proof is now complete.
\end{proof}

\bigskip

\noindent
The theorems~\ref{th1.1}, \ref{th1.111}, \ref{th1.2} and \ref{th1} are direct consequences of Lemmas  \ref{l1.1cutoff}, \ref{l1.1cutoff'}, \ref{prop1.2},
\eqref{4.asymp}, \eqref{eig02} and the Mehler formula \eqref{fork2}.

\subsection{Proof of the results for the radially symmetric Boltzmann operator}\label{radially}
We consider the Boltzmann operator (\ref{eq1}) with Maxwellian molecules $\gamma=0$ whose cross-section satisfies the assumption \eqref{new001}. As proven in \eqref{6.n113}, $Q(g,f)\in\mathscr S(\R^d)$ when $f,g\in \mathscr S(\R^d)$.

\bigskip

\begin{lemma} \label{3.lem.kn223}
For $f,g\in \mathscr S_{r}(\R^d)$, we have
\begin{equation}\label{bob1}
\mathcal{F}\big(Q(g, f)\big)(\xi)=\int_{\val \theta\le \frac{\pi}{4}}\beta(\theta)\big[\hat g(\xi \sin \theta)\hat f(\xi\cos \theta)-\hat g(0)\hat f(\xi)
\big]d\theta,
\end{equation}
where $\beta$ is the function defined in \eqref{new001}.
\end{lemma}

\bigskip

\noindent
Notice that the integral (\ref{bob1}) is well-defined according to Lemma~\ref{new003} since the function $\hat g\in \mathscr S_{r}(\R^d)$ is even.

\bigskip

\begin{proof}
Thanks to the Bobylev formula (Proposition~\ref{ao1}), we may write with $\nu=\frac{\xi}{\val \xi}$,
\begin{multline}\label{3.kkww2}
\mathcal{F}\big(Q(g, f)\big)(\xi)=\int_{(0,\pi)_{\theta}\times\mathbb S^{d-2}_{\omega}}
b(\cos \theta)(\sin \theta)^{d-2} \\ 
\times \Big[\hat g\left(\frac{\xi-\val \xi( \omega\sin \theta\oplus \nu\cos \theta)}2\right)\hat f\left(\frac{\xi+\val \xi( \omega\sin \theta\oplus \nu\cos \theta)}2\right)-\hat g(0)\hat f(\xi)\Big]d\theta d\omega.
\end{multline}
The cross section $b(\cos \theta)$ is supported where $0\le \theta\le \frac{\pi}{2}$ and we notice that
\begin{align*}
\xi-\val \xi( \omega\sin \theta\oplus \nu\cos \theta)=& \ \val\xi\bigl(-\omega\sin\theta\oplus\nu(1-\cos \theta)\bigr)\\
= & \ 2\val \xi\sin\Big(\frac\theta2\Big)\Big[-\omega\cos\Big(\frac\theta2\Big)\oplus\nu\sin\Big(\frac\theta2\Big)\Big],
\end{align*}
\begin{align*}
\xi+\val \xi( \omega\sin \theta\oplus \nu\cos \theta)= & \ \val\xi\bigl(\omega\sin\theta\oplus\nu(1+\cos \theta)\bigr)\\
= & \ 2\val \xi\cos\Big(\frac\theta2\Big)\Big[\omega\sin\Big(\frac\theta2\Big)\oplus\nu\cos\Big(\frac\theta2\Big)\Big],
\end{align*}
so that, since $\hat g, \hat f$ are radial functions,
$$\hat g\Big(\frac{\xi-\val \xi( \omega\sin \theta\oplus \nu\cos \theta)}2\Big)=\hat g\Big(\val \xi\sin\Big(\frac\theta2\Big) \nu\Big)=\hat g\Big(\xi\sin\Big(\frac\theta2\Big)\Big),$$
$$\hat f\Big(\frac{\xi+\val \xi( \omega\sin \theta\oplus \nu\cos \theta)}2\Big)=\hat f\Big(\val \xi\cos\Big(\frac\theta2\Big) \nu\Big)=\hat f\Big(\xi\cos\Big(\frac\theta2\Big)\Big),$$
yielding
{\small \begin{align*}
\mathcal{F}\big(Q(g, f)\big)&(\xi)=\val{\mathbb S^{d-2}}\int_{0}^{\frac{\pi}{2}}b(\cos \theta)(\sin \theta)^{d-2}\Big[\hat g\Big(\xi\sin\Big(\frac\theta2\Big)\Big)
\hat f\Big(\xi\cos\Big(\frac\theta2\Big)\Big)-\hat g(0)\hat f(\xi)\Big]d\theta\\
&=2\int_{0}^{\frac{\pi}{4}}\underbrace{b(\cos2\theta)(\sin2\theta)^{d-2}  \val{\mathbb S^{d-2}}}_{=\beta(\theta)\text{ from \eqref{new001}}}\Big[
\hat g(\xi\sin\theta)\hat f(\xi\cos\theta)-\hat g(0)\hat f(\xi)\Big]d\theta,
\end{align*}}
which provides \eqref{bob1}.
\end{proof}

\bigskip

\noindent
We consider the first part of the linearized non-cutoff Boltzmann operator with Maxwellian molecules
$$\mathscr{L}_{1}f=-\mu^{-1/2}Q(\mu,\mu^{1/2}f),$$
where $\mu$ is the Maxwellian distribution defined in \eqref{maxwe}. The next lemmas are analogous to Lemmas \ref{l1.1cutoff}, \ref{l1.1cutoff'},
and their proofs follow the same lines, using Lemma \ref{3.lem.kn223} instead of Lemma \ref{prop2}.

\bigskip

\begin{lemma}\label{l1.1bis}
When acting on the function space $\mathscr{S}_{r}(\rr^d)$, the operator
$\mathscr L_{1}$ is equal to the operator $\mathcal{L}_{1}$ defined by the Weyl quantization of the symbol 
\begin{equation}\label{310n}
l_{1;d}(v,\xi)=\int_{ |\theta| \leq \frac{\pi}{4}}\beta(\theta)
\left[1-\frac{\exp\big(-2 \tan^2({\frac{\theta}{2}})(|\xi|^2+\frac{|v|^2}{4})\big)}{\cos^{2d}(\frac{\theta}{2})}\right]d\theta,
\end{equation}
where $\beta$ is the function defined in \eqref{new001}.
\end{lemma}

\bigskip

\begin{lemma}\label{prop3.1bis} 
When the function $\beta$ is given by \eqref{new001'}, the symbol $l_{1;d}(v,\xi)$ belongs to the class $\mathbf{S}^{s}(\rr^{2d})$. 
Furthermore, there exists a sequence of real numbers $(c_{k,d})_{k \geq 1}$ such that for all $N \geq 1$,
$$l_{1;d}(v,\xi)\equiv c_0 \Big(1+|\xi|^2+\frac{|v|^2}{4}\Big)^s-d_{0}+\sum_{k=1}^Nc_{k,d}\Big(1+|\xi|^2+\frac{|v|^2}{4}\Big)^{s-k} \mod \mathbf{S}^{s-N-1}(\rr^{2d}),$$
where the positive constants $c_{0}, d_0>0$ are given by \eqref{czero} and \eqref{dzero}.
\end{lemma}

\bigskip

\begin{proof}[Proof of Lemmas \ref{l1.1bis}-\ref{prop3.1bis}]
The following proofs are very similar to those given for the non-cutoff Kac operator. However, we pay attention to slightly different computational details due to the multidimensional situation.
Let $u$ be in the  space $\mathscr{S}_{r}(\R^{d})$.
It follows from the Bobylev formula (Lemma \ref{3.lem.kn223}) and the Fourier inversion formula that
\begin{multline*}-\mu^{-1/2}Q(\mu,\mu^{1/2}u)(v)\\=
\frac{e^{\frac{|v|^2}{4}}}{(2 \pi)^{\frac{3d}{4}}}\iint_{\R^{d}\times(-\frac\pi{4},\frac\pi{4})}\beta(\theta)
 \left[\widehat \mu(0)\widehat{\mu^{1/2}u}(\eta)-\widehat{\mu}(\eta \sin{\theta})\widehat{\mu^{1/2}u}(\eta\cos{\theta})\right]e^{iv \cdot \eta}d\eta d\theta.
\end{multline*}
By using \eqref{eq4}, we may write
$$-\mu^{-1/2}Q(\mu,\mu^{1/2}u)(v) =\int_{\val \theta\le \frac{\pi}{4}}\beta(\theta)(\mathcal L_{1,\theta} u)(v) d\theta,$$
where the distribution-kernel of the operator $\mathcal L_{1,\theta}$ is  given by
\begin{align*}
\mathfrak{L}_{1,\theta}(v,y)=\delta_{0}(v-y)-\frac{e^{\frac{\val v^2-\val y^2}{4}}}{(2\pi)^{\frac{d}{2}} \val{\sin\theta}^{d}}\exp{-\frac{\val{v-y\cos \theta}^2}{2\sin^2\theta}}.
\end{align*}
Since 
$$\mathfrak{L}_{1,\theta}\Big(v-\frac y2,v+\frac y2\Big)=\delta_{0}(y)-\frac{e^{-\frac{v \cdot y}{2}}}{(2\pi)^{\frac{d}{2}} \val{\sin\theta}^{d}}\exp{-\frac{\val{v-\frac y2-(v+\frac y2)\cos \theta}^2}{2\sin^2\theta}},$$
we deduce from \eqref{4.lk555} that the Weyl symbol $l_{1,\theta;d}$ of the operator $\mathcal L_{1,\theta}$ writes as
$$l_{1,\theta;d}(v,\xi)=1-\ell_{1,\theta;d}(v,\xi),$$
where 
\begin{align*}
\ell_{1,\theta;d}(v,\xi)&=
\frac{1}{(2\pi)^{\frac{d}{2}} \val{\sin\theta}^{d}}
e^{-\frac{\val v^2\sin^4(\frac{\theta}{2})}{2\sin^2(\frac{\theta}{2})\cos^2(\frac{\theta}{2})}}\int_{\rr^d} e^{iy \cdot \xi}\exp{-\frac{\val y^2\cos^4(\frac{\theta}{2})}{8\sin^2(\frac{\theta}{2})\cos^2(\frac{\theta}{2})}} dy\\
&=\frac{1}{\cos^{2d}(\frac{\theta}{2})}\exp{-\Big(2\tan^2\Big(\frac{\theta}{2}\Big)\Big(\val\xi^2+\frac{\val v^2}4\Big)\Big)}.
\end{align*}
This leads to
$$l_{1,\theta;d}(v,\xi)=1-\frac{\exp{-\big(2\tan^2(\frac{\theta}{2})(\val \xi^2+\frac{\val v^2}4)\big)}}{ \cos^{2d}(\frac{\theta}{2})},$$
According to Lemma~\ref{new003}, this proves \eqref{310n} since $l_{1,\theta;d}(v,\xi)$ is an even, smooth function of the variable $\theta$ on the interval $[-\frac{\pi}{4},\frac{\pi}{4}]$ which vanishes at $\theta=0$.
We shall now check that the symbol $l_{1;d}$ belongs to the class ${\mathbf S}^s(\R^{2d})$. Setting
$$\lambda=1+|\xi|^2+\frac{|v|^2}{4},$$
we use the substitution rule with $\tau=\tan^2(\frac{\theta}{2})$ in the formula \eqref{310n} to get that
\begin{align*}
l_{1;d}(v,\xi)&=2\int_{0}^{\tan^2(\frac{\pi}{8})}\frac{\cos(\frac{\theta}{2})}{\sin^{1+2s}(\frac{\theta}{2})}\bigl(1-(1+\tau)^{d}e^{-2\tau (\lambda-1)}\bigr)\frac{d\tau}{(1+\tau)\tan(\frac{\theta}{2})}\\
&=2\int_{0}^{\tan^2(\frac{\pi}{8})}\underbrace{\tau^{-1-s}}_{u'(\tau)}\underbrace{\bigl((1+\tau)^{s-1}-(1+\tau)^{s+d-1} e^{-2\tau (\lambda-1)}
\bigr) }_{v(\tau)}d\tau,
\end{align*}
since $d\tau=\tan(\frac{\theta}{2})\bigl(1+\tan^2(\frac{\theta}{2})\bigr)d\theta$.
Integrating by parts, we obtain that
\begin{align*}
& \ l_{1;d}(v,\xi)=\frac{2}{s\tan^{2s}(\frac{\pi}{8})}\Big[\Big(1+\tan^2\Big(\frac{\pi}{8}\Big)\Big)^{s+d-1} \overbrace{e^{-2(\lambda-1)\tan^2(\frac{\pi}{8})}}^{\in {\mathbf S}^{-\io}} -\Big(1+\tan^2\Big(\frac{\pi}{8}\Big)\Big)^{s-1}\Big]
\\ & \ +\frac2s\int_{0}^{\tan^2(\frac{\pi}{8})} \big[(s-1)(1+\tau)^{s-2}-(s+d-1)(1+\tau)^{s+d-2}e^{-2\tau (\lambda-1)}
\\ & \hspace{8cm} +2(1+\tau)^{s+d-1} (\lambda-1) e^{-2\tau (\lambda-1)}\big] \frac{d\tau}{\tau^{s}}
\\ & \hspace{1.55cm} =-\frac{2(1+\tan^2(\frac{\pi}{8}))^{s-1}}{s\tan^{2s}(\frac{\pi}{8})}+\wt{l_{1;d}}(v,\xi)+{\mathbf S}^{-\io},
\end{align*}
where
\begin{align*}
\wt{l_{1;d}}(v,\xi)&=\frac2s\int_{0}^{\lambda\tan^2(\frac{\pi}{8})}\lambda^{s-1}\Bigl[(s-1)\Big(1+\frac\sigma\lambda\Big)^{s-2}\\
& \quad -(s+d-1)\Big(1+\frac\sigma\lambda\Big)^{s+d-2}e^{-2\sigma}e^{\frac{2\sigma}{\lambda}}+2\Big(1+\frac\sigma\lambda\Big)^{s+d-1} (\lambda-1) e^{-2\sigma}e^{\frac{2\sigma}{\lambda}}\Bigr] \frac{d\sigma}{\sigma^{s}}\\
&=\frac{2(s-1)}{s}\lambda^{s-1}\int_{0}^{\lambda\tan^2(\frac{\pi}{8})}\Big(1+\frac\sigma\lambda\Big)^{s-2} \frac{d\sigma}{\sigma^{s}}\\
 & \quad  +\frac{4}{s}\lambda^{s-1}(\lambda-1)\int_{0}^{\lambda\tan^2(\frac{\pi}{8})}\Big(1+\frac\sigma\lambda\Big)^{s+d-1}e^{-2\sigma} e^{\frac{2\sigma}{\lambda}}\frac{d\sigma}{\sigma^{s}}\\
& \quad  -\frac{2(s+d-1)}{s}\lambda^{s-1}\int_{0}^{\lambda\tan^2(\frac{\pi}{8})}\Big(1+\frac\sigma\lambda\Big)^{s+d-2}e^{-2\sigma} e^{\frac{2\sigma}{\lambda}}\frac{d\sigma}{\sigma^{s}}.
\end{align*}
The sum of the first and last terms writes as
\begin{multline}\label{4.l11te'}
l_{1,1;d}(v,\xi)\\ =2\lambda^{s-1}
\int_{0}^{\lambda\tan^2(\frac{\pi}{8})}\Bigl[\frac{s-1}{s}-\frac{s+d-1}{s}\Big(1+\frac\sigma\lambda\Big)^{d}e^{-2\sigma}e^{\frac{2\sigma}{\lambda}}\Bigr]\Big(1+\frac\sigma\lambda\Big)^{s-2}\frac{d\sigma}{\sigma^{s}}.
\end{multline}
The main term is the second one
\begin{equation}\label{4.l12te'}
l_{1,2;d}(v,\xi)=\frac{2^{1+s}}{s}\lambda^{s-1}(\lambda-1)\int_{0}^{2\lambda\tan^2(\frac{\pi}{8})}\Big(1+\frac{w}{2\lambda}\Big)^{s+d-1}e^{\frac{w}{\lambda}}e^{-w} \frac{dw}{w^{s}}.
\end{equation}
We may write
\begin{equation}\label{4.kjn11'}
l_{1;d}(v,\xi)= -\frac{2(1+\tan^2(\frac{\pi}{8}))^{s-1}}{s\tan^{2s}(\frac{\pi}{8})} +{l_{1,1;d}}(v,\xi)+{l_{1,2;d}}(v,\xi)+{\mathbf S}^{-\io}.
\end{equation}
By using that $\tan^2(\frac{\pi}{8})<1$ and that the function
$$z\mapsto \kappa_{d}(z)=(1+z)^{s+d-1} e^{2z}=\sum_{j\ge 0}a_{j,d}z^j,$$ 
is holomorphic on $\val z<1$, we obtain that  
\begin{align*}
l_{1,2;d}&=\frac{2^{1+s}}{s}\lambda^{s-1}(\lambda-1)\sum_{j\ge 0} \frac{a_{j,d}}{2^{j}\lambda^{j}}\int_{0}^{2\lambda\tan^2(\frac{\pi}{8})}w^{j-s}e^{-w}dw\\
&=\frac{2^{1+s}}{s}\lambda^{s-1}(\lambda-1)\sum_{0\le j\le N} \frac{a_{j,d}}{2^{j}\lambda^{j}}\int_{0}^{2\lambda\tan^2(\frac{\pi}{8})}w^{j-s}e^{-w}dw\\
&\hskip33pt+\frac{2^{1+s}}{s}\frac{\lambda^{s-1}(\lambda-1)}{(2\lambda)^{N+1}}\int_{\rho=0}^1\int_{w=0}^{2\lambda\tan^2(\frac{\pi}{8})}\frac{(1-\rho)^N}{N!}
\kappa_{d}^{(N+1)}\Big(\frac{\rho w}{2\lambda}\Big)w^{N+1-s}e^{-w} d\rho dw\\
&=\frac{2^{1+s}}{s}\lambda^{s-1}(\lambda-1)\sum_{0\le j\le N} \frac{a_{j,d}}{2^{j}\lambda^{j}}\Gamma(1+j-s)\\
&\hskip22pt-\underbrace{\frac{2^{1+s}}{s}\lambda^{s-1}(\lambda-1)\sum_{0\le j\le N} \frac{a_{j,d}}{2^{j}\lambda^{j}}\int_{2\lambda\tan^2(\frac{\pi}{8})}^{+\io}w^{j-s}e^{-w}dw}_{\in {\mathbf S}^{-\io}}\\
&\hskip22pt+\frac{2^{1+s}}{s}\frac{\lambda^{s-1}(\lambda-1)}{(2\lambda)^{N+1}}\int_{\rho=0}^1\int_{w=0}^{2\lambda\tan^2(\frac{\pi}{8})}\frac{(1-\rho)^N}{N!}
\kappa_{d}^{(N+1)}\Big(\frac{\rho w}{2\lambda}\Big)w^{N+1-s}e^{-w} d\rho dw.
\end{align*}
To prove that the last line belongs to ${\mathbf S}^{s-N-1}$, it is sufficient to prove that
$$\omega_{0;d}(v,\xi)=\int_{\rho=0}^1\int_{w=0}^{2\lambda\tan^2(\frac{\pi}{8})}\frac{(1-\rho)^N}{N!}\kappa_{d}^{(N+1)}\Big(\frac{\rho w}{2\lambda}\Big)w^{N+1-s}e^{-w} d\rho dw\in {\mathbf S}^0.$$
To that end, we first notice that the function $\omega_{0;d}$ is bounded
\begin{multline*}
\val{\omega_{0;d}(v,\xi)}\le\int_{\rho=0}^1\int_{w=0}^{2\lambda\tan^2(\frac{\pi}{8})}\frac{(1-\rho)^N}{N!}\norm{\kappa_{d}^{(N+1)}}_{L^\io(\val z\le\tan^2(\frac{\pi}{8})) }w^{N+1-s}e^{-w} d\rho dw \\
\le \frac{\Gamma(N+2-s)}{(N+1)!}\norm{\kappa_{d}^{(N+1)}}_{L^\io(\val z\le\tan^2(\frac{\pi}{8})) }.
\end{multline*}
Writing $\omega_{0;d}(v,\xi)=\Omega_{0;d}(\lambda(v,\xi))$, we have
\begin{multline*}
\frac{d\Omega_{0;d}}{d\lambda}=\Big(2\lambda\tan^2\Big(\frac\pi8\Big)\Big)^{N+2-s}\frac{e^{-2\lambda\tan^2(\frac\pi8)}}{\lambda}\int_{0}^1\frac{(1-\rho)^N}{N!}\kappa_{d}^{(N+1)}\Bigl(\rho\tan^2\Big(\frac\pi8\Big)\Bigr)d\rho\\
-\lambda^{-1}\int_{\rho=0}^1\int_{w=0}^{2\lambda\tan^2(\frac\pi8)}\frac{(1-\rho)^N}{N!}\kappa_{d}^{(N+2)}\Big(\frac{\rho w}{2\lambda}\Big)\frac{\rho w}{2\lambda}w^{N+1-s}e^{-w} dw,
\end{multline*}
where the first line belongs to ${\mathbf S}^{-\io}$, and by following the exact same reasoning as for bounding the function $\omega_{0;d}$, we notice that the second line is bounded above in modulus by a constant times $\lambda^{-1}$. It follows that 
$$\val{\nabla_{v,\xi}\omega_{0;d}}\lesssim  \lambda^{-1}\val{\nabla_{v,\xi} \lambda}\lesssim \lambda^{-1/2}.$$
The higher-order derivatives may be handled in the very same way. It follows that 
for any $N \in \nn$,
\begin{equation}\label{4.14kkb'}
l_{1,2;d}\equiv\sum_{0\le j\le N}\lambda^{s-j-1}(\lambda-1)\frac{2^{1+s-j}}{s}\Gamma(1+j-s)a_{j}\mod {\mathbf S}^{s-N-1}.
\end{equation}
The study of the term $l_{1,1;d}$ is very similar. We may write   
\begin{multline} \label{4.nb556'}
l_{1,1;d}(v,\xi)=2^{s}\frac{s-1}{s}\lambda^{s-1}\int_{0}^{2\lambda\tan^2(\frac{\pi}{8})}\Big(1+\frac w{2\lambda}\Big)^{s-2}\frac{dw}{w^{s}}\\
-2^{s}\frac{s+d-1}{s}\lambda^{s-1}\int_{0}^{2\lambda\tan^2(\frac{\pi}{8})}\Big(1+\frac w{2\lambda}\Big)^{s+d-2}e^{\frac{w}{\lambda}}e^{-w}\frac{dw}{w^{s}},
\end{multline}
so that the last term is almost identical to the symbol $l_{1,2;d}$ (with leading term $\lambda ^{s-1}$) and the first integral in the last line is equal to the negative constant
$$-\frac{2(1-s)}{s}\int_{0}^{3-2^{3/2}}(1+t)^{s-2}\frac{dt}{t^{s}}.$$
We deduce from \eqref{4.kjn11'} and \eqref{4.14kkb'} that
\begin{equation}\label{4.asymp'}
l_{1;d}\equiv \frac{2^{1+s}}{s}\Gamma(1-s) \lambda^s-d_{0}+\sum_{1\le j\le N}c_{j,d}
 \lambda^{s-j}\mod {\mathbf S}^{s-N-1},
\end{equation}
where
$d_{0}$ is the constant given in \eqref{dzero}. This ends the proof of Lemma \ref{prop3.1bis}.
\end{proof}

\bigskip

\noindent
We consider the second part of the linearized non-cutoff Boltzmann operator with Maxwellian molecules 
$$\mathscr{L}_{2}u= -\mu^{-1/2}Q(\mu^{1/2}u,\mu),$$
where $\mu$ is the Maxwellian distribution defined in (\ref{maxwe}).

\bigskip

\begin{lemma}\label{prop1.2'} 
When acting on $\mathscr S_{r}(\R^{d})$, the second part of the linearized non-cutoff Boltzmann operator with Maxwellian molecules $\mathscr L_{2}$ is equal to the operator $\mathcal L_{2}$ defined by the Weyl quantization of the symbol
\begin{multline}\label{4.oper2'}
l_{2;d}(v,\xi)=
\int_{|\theta| \leq \frac{\pi}{4}}
\beta(\theta)\Biggr[2^{d} e^{-2(\val\xi^2+\frac{\val v^2}{4})}-2^{d-1}
\frac{\exp\Big(-\frac{2\cos^2 \theta(\val\xi^2+\frac{\val v^2}{4})}{(1+\sin{\theta})^2}\Big)}{(1+\sin \theta)^{d}}
\\-2^{d-1}\frac{\exp\Big(-\frac{2\cos^2\theta(\val\xi^2+\frac{\val v^2}{4})}{(1-\sin{\theta})^2}\Big)}{(1-\sin \theta)^{d}}
\Biggr]d\theta,
\end{multline}
satisfying
$$\forall (\alpha,\beta) \in \N^{2d}, \exists C_{\alpha,\beta}>0, \forall (v,\xi)\in \R^{2d}, \ |\partial_v^{\alpha}\partial_{\xi}^{\beta}l_{2;d}(v,\xi)| \leq C_{\alpha,\beta}e^{-\frac{1}{3}(\val\xi^2+\frac{\val v^2}{4})},$$
and implying in particular that $l_{2;d} \in \mathbf{S^{-\infty}}(\R^{2d})$.
\end{lemma}

\bigskip

\noindent
Notice that the integral \eqref{4.oper2'} makes sense as an ordinary integral according to Lemma~\ref{new003}.

\bigskip

\begin{proof}
As in the previous section, we deduce from the Bobylev formula (Lem\-ma \ref{3.lem.kn223}) that for any $u\in \mathscr S_{r}(\R^d)$,
\begin{multline*}\mathscr{L}_{2}u=\\
\frac{e^{\frac{\val v^2}{4}}}{(2 \pi)^{\frac{3d}{4}}}\int_{|\theta| \leq \frac{\pi}{4}}\beta(\theta)\left(\int_{\R^d}\left[(\widehat{\mu^{1/2}\breve u})(0)\widehat{\mu}(\eta) -(\widehat{\mu^{1/2}\breve u})(\eta \sin{\theta})\widehat{\mu}(\eta\cos{\theta})\right]e^{iv \cdot \eta}d\eta\right) d\theta.
\end{multline*}
It follows that 
\begin{align*}
& \ (\mathscr L_{2} u)(v)\\
= & \  \iint_{\R^{d}\times(-\frac\pi{4},\frac\pi{4})}\beta(\theta)\left(\frac{1}{(2 \pi)^{d}}\int_{\R^{d}}e^{\frac{\val v^2-\val y^2}{4}}
\left[e^{-\frac{\val \eta^2}{2}}-e^{-\frac{\val\eta^2\cos^2 \theta}{2}}e^{-iy\cdot \eta\sin{\theta}}\right]e^{iv\cdot \eta}\breve u(y) dy\right)d\eta d\theta\\
= & \ \int_{\val \theta\le \frac{\pi}{4}}\beta(\theta)(\mathcal L_{2,\theta;d} u)(v) d\theta,
\end{align*}
where the distribution-kernel of the operator $\mathcal L_{2,\theta;d}$ is given by (see subsection \ref{6.sec.susub}),
\begin{equation}\label{4.cx229'}
\frac12\left(\mathfrak{L}_{2,\theta;d}(v,y)+\mathfrak{L}_{2,\theta;d}(v,-y)\right),
\end{equation}
whereas the oscillatory integral $\mathfrak{L}_{2,\theta;d}$ is
$$\mathfrak{L}_{2,\theta;d}(v,y)=\frac{e^{\frac{\val v^2-\val y^2}{4}}}{(2 \pi)^{d}}\int_{\R^{d}} \left[e^{-\frac{\val \eta^2}{2}}-e^{-\frac{\val\eta^2\cos^2 \theta}{2}}e^{-iy \cdot \eta\sin{\theta}}\right]e^{iv \cdot \eta}d\eta.$$
By using \eqref{eq4}, we find that
\begin{align*}
\mathfrak{L}_{2,\theta;d}(v,y)&= \frac{e^{\frac{\val v^2-\val y^2}{4}}}{(2 \pi)^{d}}\Bigl((2\pi)^{\frac{d}{2}} e^{-\frac{\val v^2}{2}}-\frac{(2\pi)^{\frac{d}{2}} }{\cos^{d} \theta}\exp-\frac{\val{v-y\sin \theta}^2}{2\cos^2 \theta}\Bigr)\\
&=(2\pi)^{-\frac{d}{2}} \Bigl[e^{-\frac{\val y^2+\val v^2}{4}}-\frac{1}{\cos^{d} \theta}\exp-\Bigl(\frac{\val{v-y\sin \theta}^2}{2\cos^2 \theta}+\frac{\val y^2-\val v^2}{4}\Bigr)\Bigr].
\end{align*}
We obtain that
\begin{multline*}
(2\pi)^{\frac{d}{2}}\mathfrak{L}_{2,\theta;d}\Big(v-\frac y2,v+\frac y2\Big)\\=e^{-\frac{\val y^2+4\val v^2}{8}}-\frac{1}{\cos^{d} \theta}\exp-\Bigl(\frac{4(1-\sin \theta)^2\val v^2+(1+\sin \theta)^2\val y^2}{8\cos^2 \theta}\Bigr)
\end{multline*}
and
\begin{multline*}
(2\pi)^{\frac{d}{2}}\mathfrak{L}_{2,\theta;d}\Big(v-\frac y2,-v-\frac y2\Big)\\ =e^{-\frac{\val y^2+4\val v^2}{8}}
 -\frac{1}{\cos^{d} \theta}\exp-\Bigl(\frac{4(1+\sin \theta)^2\val v^2+(1-\sin \theta)^2\val y^2}{8\cos^2 \theta}\Bigr).
\end{multline*}
Setting
$$l_{2,+,\theta;d}(v,\xi)=\int_{\rr^d}  \mathfrak{L}_{2,\theta;d}\Big(v-\frac y2,v+\frac y2\Big)e^{iy \cdot \xi} dy,$$
it follows from \eqref{eq4} that 
\begin{align*}
l_{2,+,\theta;d}(v,\xi)= & \ 2^{d}e^{-2(\val \xi^2+\frac{\val v^2}4)}-\frac{2^d}{(1+\sin \theta)^{d}}\exp-\Big(\frac{(1-\sin\theta)^2\val v^2}{2\cos^2\theta}+
\frac{2(1-\sin\theta)\val \xi^2}{1+\sin\theta}\Bigr)\\
= & \ 2^{d}e^{-2(\val\xi^2+\frac{\val v^2}4)}-\frac{2^d}{(1+\sin \theta)^{d}}\exp-\Big(\frac{2(1-\sin\theta)}{1+\sin\theta}\Big(\val \xi^2+\frac{\val v^2}4\Big)\Big).
\end{align*}
We deduce from \eqref{4.cx229'} that the Weyl symbol of the operator $\mathcal L_{2,\theta;d}$ is given by
\begin{align*}
l_{2,\theta;d}(v,\xi)&=\frac{1}{2}\int_{\rr^d}  \mathfrak{L}_{2,\theta}\Big(v-\frac y2,v+\frac y2\Big)e^{iy\cdot \xi} dy+\frac{1}{2}\int_{\rr^d}  \mathfrak{L}_{2,\theta}\Big(v-\frac y2,-v-\frac y2\Big)e^{iy\cdot \xi} dy\\
&=\frac{1}{2}\big(l_{2,+,\theta;d}(v,\xi)+l_{2,+,-\theta;d}(v,\xi)\big).
\end{align*}
This implies that
\begin{multline*}
l_{2,\theta;d}(v,\xi)=2^{d}e^{-2(\val \xi^2+\frac{\val v^2}4)}-\frac{2^{d-1}}{(1+\sin \theta)^{d}}\exp-\Big(\frac{2(1-\sin\theta)}{1+\sin\theta}(\val \xi^2+\frac{\val v^2}4)\Big).\\
-\frac{2^{d-1}}{(1-\sin \theta)^{d}}\exp-\Big(\frac{2(1+\sin\theta)}{1-\sin\theta}(\val \xi^2+\frac{\val v^2}4)\Big).
\end{multline*}
The end of the proof of Lemma~\ref{prop1.2'} is then identical to the proof given for Lemma~\ref{prop1.2}.
\end{proof}

\bigskip

\noindent
Theorems~\ref{th1.1bis}, \ref{th1.111bis}, \ref{th1bis} and Corollary~\ref{th1.111bisbisbis} are direct consequences of Lemmas \ref{l1.1bis}, \ref{prop3.1bis}, \ref{prop1.2'} and the formulas \eqref{6.444mm}, \eqref{fork2} along with \eqref{6.asyvp} and (\ref{eig02}).

\section{Appendix}\label{appendix}

\subsection{A distribution of order 2}
For $\phi$ a function defined on $\R$, we denote
\begin{equation}\label{6.evenp}
\breve{\phi} (\theta)=\frac12\bigl(\phi(\theta)+\phi(-\theta)\bigr),
\end{equation}
its even part.

\bigskip

\begin{lemma}\label{new003}
Let $\nu\in L^1_{loc}({\rr^*})$ be an even function such that $\theta^2\nu(\theta)\in L^1(\rr)$. Then, the mapping
$$ \phi \in C^2_{c}(\rr) \mapsto\lim_{\varepsilon\rightarrow 0_{+}}\int_{\vert\theta\vert\ge \varepsilon}
\nu(\theta)\bigl(\phi(\theta)-\phi(0)\bigr) d\theta=\int_{0}^1\int_{\rr}(1-t)\theta^2\nu(\theta)\phi''(t\theta) d\theta dt,$$
is defining a distribution of order 2 denoted $\finp{(\nu)}$. The linear form $\finp{(\nu)}$ can be extended to $C^{1,1}$ functions
($C^1$ functions whose second derivative is $L^\io$).
For $\phi\in C^{1,1}$ satisfying $\phi(0)=0$, the function $\nu \breve\phi$ belongs to $L^1(\R)$ and
\begin{equation}\label{6.11ssz}
\poscal{\finp{(\nu)}}{\phi}=\int\nu(\theta)\breve{\phi}(\theta)d\theta.
\end{equation}
\end{lemma}

\bigskip

\begin{proof}
We have
$$\int_{\vert\theta\vert\ge \varepsilon}
\nu(\theta)\bigl(\phi(\theta)-\phi(0)\bigr) d\theta=\int_{0}^1\int_{\vert\theta\vert\ge \varepsilon} (1-t) \theta^2\nu(\theta)\phi''(t\theta) d\theta dt,$$
and the Lebesgue dominated convergence theorem gives the first result.
The extension to $C^{1,1}$ follows from the formula
$$\frac12(\phi(\theta)-\phi(0))+\frac12(\phi(-\theta)-\phi(0))=\frac12\int_{0}^\theta\bigl(\phi'(\tau)-\phi'(-\tau)\bigr) d\tau,$$
since the absolute value of the latter is bounded above by $\frac{1}{2}\norm{\phi''}_{L^\io}\theta^2$. This implies that 
$$\nu(\theta)\times\text{\tt even part}(\phi(\theta)-\phi(0))\in L^1,$$ 
proving as well the last statement.
\end{proof}

\subsection{The non-cutoff Kac and Boltzmann collision operators}

\subsubsection{The Kac collision operator}\label{kacsection}
Let $g,f \in \mathscr{S}(\rr)$ be Schwartz functions. We define
\begin{equation}\label{6.kac01}
F_{f,g}(\underbrace{v,v_{*}}_{w})= f(v) g(v_{*}),\quad
\phi_{f,g}(\theta,v)=\int_{\rr}\bigl(F_{f,g}(R_{\theta}w)-F_{f,g}(w)\bigr) dv_{*},
\end{equation}
where $R_{\theta}$ stands for the rotation of angle $\theta$ in $\R^2$,
$$R_{\theta}=\mat22{\cos\theta}{-\sin\theta}{\sin\theta}{\cos\theta}=\exp(\theta J),\quad J=R_{\frac{\pi}{2}}.$$
We have
$$F_{f,g}(R_{\theta}w)-F_{f,g}(w)=f(v\cos\theta - v_*\sin\theta)g(v\sin\theta +v_*\cos\theta)-f(v) g(v_{*}),$$
so that by using the notations $f_*'=f(v_*')$, $f'=f(v')$, $f_*=f(v_*)$, $f=f(v)$ with
$$v'=v\cos\theta - v_*\sin\theta, \quad v'_* =v\sin\theta +v_*\cos\theta, \quad v,v_* \in \rr,$$
we may write
\begin{equation}\label{6.fuphi}
\phi_{f,g}(\theta,v)=\int_{\R}(g'_{*}f'-g_{*}f)dv_{*}.
\end{equation}
Furthermore, we easily check that its even part as a function of the variable $\theta$ is given by
\begin{equation}\label{6.ev0np}
\breve{\phi}_{f,g}(\theta,v)=\int_{\R}\bigl((\breve{g})'_{*}f'-g_{*}f\bigr) dv_{*}
=\int_{\R}\bigl((\breve{g})'_{*}f'-(\breve{g})_{*}f\bigr) dv_{*}.
\end{equation}
Note that, for each $\theta\in \R$, the mapping
$$(f,g) \in \mathscr S(\R)\times\mathscr S(\R) \mapsto \phi_{f,g}(\theta,\cdot)\in \mathscr S(\R),$$
is continuous uniformly with respect to~$\theta$. In fact, the function $F_{f,g}$ belongs to $\mathscr S(\R^2)$. By denoting $\Pi_1$ the projection onto the first variable, this implies that the function
$$v^l\p_{v}^k\phi_{f,g}(\theta,v)=\int\Pi_{1}(w)^l\p_{v}^k \Phi_{f,g}(\theta,w) dv_{*},$$
is bounded since 
$$\Phi_{f,g}(\theta,w)=F_{f,g}(R_{\theta}w)-F_{f,g}(w) \in \mathscr S(\R^2).$$ 
As a result, the function $v\mapsto \phi_{f,g}(\theta,v)$ belongs to $\mathscr S(\R)$ uniformly with respect to~$\theta$. Moreover, the second derivative with respect to~$\theta$ of the function $\Phi_{f,g}$,
$$F_{f,g}''(e^{\theta J}w)\bigl(e^{\theta J} Jw,e^{\theta J} Jw\bigr)-F_{f,g}'(e^{\theta J}w)e^{\theta J}w,$$
belongs to $\mathscr S(\R^2)$ uniformly with respect to $\theta$. This implies that the second derivative with respect to $\theta$ of the function $\phi_{f,g}$ is in $\mathscr S(\R)$ uniformly with respect to $\theta$. 

We define the non-cutoff Kac operator as
\begin{equation}\label{6.kac03}
K(g,f)(v)=\poscal{\finp(\un_{(-\frac{\pi}{4},\frac{\pi}{4})}\beta)}{\phi_{f,g}(\cdot,v)},
\end{equation}
when $\beta$ is a function satisfying \eqref{ah1}.
Since $\phi_{f,g}(0,v)\equiv 0$, Lemma \ref{new003} allows to replace the finite part by the absolutely converging integral
\begin{equation}\label{6.kac02}
K(g,f)(v)=\int_{|\theta| \leq \frac{\pi}{4}}\beta(\theta)\Bigl(\int_{\RR}\big({\breve{g}}'_* f'-{\breve{g}}_*f \big) dv_*\Bigr)d\theta=K(\breve{g},f)(v).
\end{equation}

\bigskip

\begin{lemma}\label{6.lem.defkac}
For $g,f$ be in $\mathscr{S}(\rr)$, then $K(g,f)\in \mathscr{S}(\rr).$
\end{lemma}

\bigskip

\begin{proof}
We deduce from the above properties of the function $\phi_{f,g}$ and \eqref{6.kac03} that the function $K(g,f)$ is smooth and that for any $k,l \in \nn$,
$$v^l\p_{v}^k (K(g,f))(v)=\poscal{\finp(\un_{(-\frac{\pi}{4},\frac{\pi}{4})}\beta)}{v^l\p_{v}^k\phi_{f,g}(\cdot,v)}.$$
Since the second derivative with respect to $\theta$ of the function $\phi_{f,g}$ belongs to $\mathscr S(\R)$ uniformly with respect to $\theta$, we obtain that
$$v\mapsto v^l\p_{v}^k (K(g,f))(v)\in L^\io.$$
\end{proof}

\subsubsection{The Boltzmann collision operator}
We consider the Boltzmann operator with Maxwellian molecules
$$Q(g, f)=\int_{\rr^d}\int_{\SSS^{d-1}}b\Big(\frac{v-v_{*}}{\val{v-v_{*}}}\cdot \sigma\Big)(g'_{*} f'-g_{*}f)d\sigma dv_*.$$
By using polar coordinates, $v-v_{*}=\rho\nu$, $\rho>0$, $\nu\in \mathbb S^{d-1}$, we may write
\begin{multline*}
Q(g, f)=\\
\int_{\R_{\rho}^+\times\SSS^{d-1}_{\sigma}\times\SSS^{d-1}_{\nu}}b(\nu\cdot \sigma)
\Bigl[g\Bigl(v-\frac{\rho(\sigma+ \nu)}{2}\Bigr)f\Bigl(v+\frac{\rho(\sigma- \nu)}{2}\Bigr)-g(v-{\rho\nu})f(v)\Bigr] \rho^{d-1}d\rho d\sigma d\nu.
\end{multline*}
Setting $\sigma=\omega\sin \theta\oplus\nu\cos \theta$ with $\omega\in \SSS^{d-2}$, $\omega\perp \nu$, $0<\theta<\pi,$ we obtain that
\begin{align*}
&Q(g, f)=\int_{\R_{\rho}^+\times\SSS^{d-2}_{\omega}\times (0,\pi)\times\SSS^{d-1}_{\nu}
}b(\cos \theta) \rho^{d-1} (\sin \theta)^{d-2} 
\\&\hskip44pt
\Bigl[g\Bigl(v-\frac{\rho(\omega\sin \theta\oplus\nu\cos \theta+ \nu)}{2}\Bigr)f\Bigl(v+\frac{\rho(\omega\sin \theta\oplus\nu\cos \theta- \nu)}{2}\Bigr)
\\&\hskip199pt-
g(v-{\rho\nu})f(v)
\Bigr]d\rho  d\theta d\omega d\nu
\\&\hskip33pt=
\int_{\R_{\rho}^+\times\SSS^{d-2}_{\omega}\times (0,\pi)\times\SSS^{d-1}_{\nu}
}b(\cos \theta) \rho^{d-1} (\sin \theta)^{d-2} 
\\&\hskip33pt
\Bigl[g\Bigl(v-{\rho\cos\frac{\theta}{2}\Big(\omega\sin\frac{\theta}{2}\oplus\nu\cos\frac{\theta}{2}\Big)}\Bigr)
f\Bigl(v+{\rho\sin\frac{\theta}{2}\Big(\omega\cos\frac{\theta}{2}\ominus\nu \sin\frac{\theta}{2}\Big)}\Bigr)
\\&\hskip77pt-
g(v-{\rho\nu})f(v)
\Bigr]d\rho d\theta d\omega d\nu.
\end{align*}
Since the cross section $b(\cos \theta)$ is supported where $0 \leq \theta \leq \frac{\pi}{2}$, we have
\begin{multline*}
Q(g, f)=\int_{\R_{\rho}^+\times\SSS^{d-2}_{\omega}\times (0,\frac{\pi}{4})\times\SSS^{d-1}_{\nu}}2b(\cos 2\theta) \rho^{d-1} (\sin 2\theta)^{d-2} d\rho d\theta d\omega d\nu\\
\Bigl[g\bigl(v-{\rho\cos\theta(\omega\sin\theta\oplus\nu\cos\theta)}\bigr)f\bigl(v+{\rho\sin\theta(\omega\cos\theta\ominus\nu\sin\theta)}\bigr)-g(v-{\rho\nu})f(v)\Bigr].
\end{multline*}
By using \eqref{new001}, we obtain that
\begin{multline*}
Q(g, f)=\frac{1}{\val {\SSS^{d-2}}}\int_{\R_{\rho}^+\times\SSS^{d-2}_{\omega}\times (0,\frac{\pi}{4})\times\SSS^{d-1}_{\nu}}2\beta(\theta) \rho^{d-1}d\rho d\theta d\omega d\nu\\
\Bigl[g\bigl(v-{\rho\cos\theta(\omega\sin\theta\oplus\nu\cos\theta)}\bigr)f\bigl(v+{\rho\sin\theta(\omega\cos\theta\ominus\nu\sin\theta)}\bigr)
-g(v-{\rho\nu})f(v)\Bigr].
\end{multline*}
We define 
\begin{multline}
\Psi_{f,g}(\theta,v)=\frac{1}{\val {\SSS^{d-2}}}\int_{\SSS^{d-2}_{\omega}\times \R^+_{\rho}\times\mathbb S^{d-1}_{\nu}}
g\bigl(v-{\rho\cos\theta(\omega\sin\theta\oplus\nu\cos\theta)}\bigr)
\\ \times f\bigl(v+{\rho\sin\theta(\omega\cos\theta\ominus\nu\sin\theta)}\bigr) \rho^{d-1}d\omega d\rho d\nu.
\end{multline}
We notice that
\begin{align*}
\Psi_{f,g}(-\theta,v)&=\frac{1}{\val {\SSS^{d-2}}}\int_{\SSS^{d-2}_{\omega}\times \R^+_{\rho}\times\mathbb S^{d-1}_{\nu}}g\bigl(v-{\rho\cos\theta(-\omega\sin\theta\oplus\nu\cos\theta)}\bigr)\\ 
&\hskip99pt \times f\bigl(v-{\rho\sin\theta(\omega\cos\theta\oplus\nu\sin\theta)}\bigr) \rho^{d-1}d\omega d\rho d\nu\\
&=\frac{1}{\val {\SSS^{d-2}}}\int_{\SSS^{d-2}_{\omega}\times \R^+_{\rho}\times\mathbb S^{d-1}_{\nu}}g\bigl(v-{\rho\cos\theta(\omega\sin\theta\oplus\nu\cos\theta)}\bigr)\\
&\hskip99pt  \times f\bigl(v+{\rho\sin\theta(\omega\cos\theta\ominus\nu\sin\theta)}\bigr)\rho^{d-1} d\omega d\rho d\nu\\
&=\Psi_{f,g}(\theta,v),
\end{align*}
so that the function $\theta\mapsto\Psi_{f,g}(\theta,v)$ is even. Furthermore, we have
\begin{equation}\label{6.n112}
\Psi_{f,g}(0,v)=\int_{\R^+_{\rho}\times\mathbb S^{d-1}_{\nu}}g(v-\rho\nu)f(v)\rho^{d-1} d\rho d\nu.
\end{equation}
When $f,g\in \mathscr S(\R^d)$, we get that $v \in \R^d \mapsto\p_{\theta}^m\Psi_{f,g}(\theta,v)$ belongs uniformly to $\mathscr S(\R^d)$ since
\begin{multline*}
\val{v'_{*}}^2+\val{v'}^2=\val{v-\rho\cos\theta(\omega\sin\theta\oplus\nu\cos\theta)}^2+\val{v+\rho\sin\theta(\omega\cos\theta\ominus\nu\sin\theta)}^2\\
=2\val v^2+\rho^2-2\rho v\cdot \nu=\val v^2+\val{v-\rho\nu}^2=\val{v_{*}}^2+\val{v}^2 \ge \frac{1}{3}(\val{v}^2+\rho^2).
\end{multline*}
Lemma \ref{new003} allows to define the Boltzmann operator
\begin{equation}\label{6.n113}
Q(g,f)(v)=\int_{\val \theta\le \frac{\pi}{4}}\beta(\theta) \left(\Psi_{f,g}(\theta,v)-\Psi_{f,g}(0,v) \right)d\theta.
\end{equation}
Furthermore, we have $Q(g,f)\in \mathscr S(\R^d)$ when $f,g\in \mathscr S(\R^d)$.

\subsection{The Bobylev formula}

\subsubsection{The Bobylev formula for the Boltzmann operator with Maxwellian molecules}\label{bobylev1}
For the sake of completeness, we include the statement of the Bobylev formula following the presentation given in the appendix of \cite{al-1}. The Bobylev formula provides an explicit formula for the Fourier transform of the Boltzmann operator \eqref{eq1}.

\bigskip

\begin{proposition} \label{ao1}
The Fourier transform of the Boltzmann operator with Maxwellian molecules whose cross section satisfies the assumption (\ref{sa1}), 
$$Q(g, f)(v)=\int_{\rr^d}\int_{\SSS^{d-1}}b\Big(\frac{v-v_{*}}{|v-v_{*}|} \cdot \sigma\Big)\bigl(g'_* f'-g_{*}f\bigr)d\sigma dv_*,$$
is equal to
\begin{align*}
\mathcal{F}\big(Q(g, f)\big)(\xi)= & \ \int_{\rr^d}Q(g, f)(v)e^{-iv \cdot \xi}dv\\
= & \ \int_{\SSS^{d-1}} b\Big(\frac{\xi}{|\xi|} \cdot \sigma\Big)\bigl[\widehat{g}(\xi^-)\widehat{f}(\xi^+)-
\widehat{g}(0)\widehat{f}(\xi)\bigr]d\sigma,
\end{align*}
where $\xi^+=\frac{\xi+|\xi|\sigma}{2}$ and $\xi^-=\frac{\xi-|\xi|\sigma}{2}$.
\end{proposition}

\bigskip

\subsubsection{The Bobylev formula for the Kac operator}\label{bobylev2}
For $f,g\in \mathscr S(\R)$, the function
\begin{equation}\label{6.knb55}
\theta\mapsto \psi_{f,g}(\theta,\xi)=\widehat{g}(\xi \sin{\theta})\widehat{f}(\xi\cos{\theta})-\widehat g(0)\widehat{f}(\xi),
\end{equation}
is vanishing at zero and has a bounded second derivative. According to Lemma \ref{new003}, the following integral makes sense
$$\int_{\val \theta\le \frac{\pi}{4}} \beta(\theta) \breve{\psi}_{f,g}(\theta,\xi) d\theta=\int_{\val \theta\le \frac{\pi}{4}} \beta(\theta) {\psi}_{f,\breve{g}}(\theta,\xi) d\theta,$$
where $\breve{\psi}_{f,g}(\theta,\xi)$ is the even part of the function $\psi_{f,g}$ with respect to the variable $\theta$, when the cross section $\beta$ satisfies the assumption \eqref{ah1}.

\bigskip

\begin{lemma} \label{prop2}
When the cross section satisfies the assumption \eqref{ah1}, the Kac operator $K(g,f)$ defines a Schwartz function for $f,g\in \mathscr S(\R)$.
Furthermore, its Fourier transform is given by
$$\widehat{K(g,f)}(\xi)=\int_{|\theta| \leq \frac{\pi}{4}}\beta(\theta) \left[\widehat{\breve{g}}(\xi \sin{\theta})\widehat{f}(\xi\cos{\theta})-\widehat g(0)\widehat{f}(\xi)\right]d\theta.$$
\end{lemma}

\bigskip

\begin{proof}
We deduce from Lemmas \ref{new003}, \ref{6.lem.defkac}, \eqref{6.evenp}, \eqref{6.fuphi} and (\ref{6.kac02}) that
\begin{align}\label{6.knbgg}
\widehat{K(g,f)}(\xi)= \widehat{K(\breve{g},f)}(\xi)= & \ \iint_{[-\frac\pi4,\frac{\pi}4]\times \R}\underbrace{\beta(\theta)\phi_{f,\breve g}(\theta, v)e^{-iv\xi}}_{\in L^1([-\frac\pi4,\frac{\pi}4]\times \R)}
 d\theta dv\\
 =& \ \lim_{\varepsilon\rightarrow 0}\iint_{\{\varepsilon\le \val \theta\le \frac{\pi}{4}\}\times \R}\beta(\theta)\phi_{f,\breve g}(\theta, v)e^{-iv\xi}d\theta dv.  \notag
\end{align}
We consider
$$\mathcal I_{\varepsilon}=\iiint_{\{\eps \leq |\theta| \leq \frac{\pi}{4}\}\times\RR^2}\beta(\theta)\breve g(v\sin\theta +v_*\cos\theta) f(v\cos\theta - v_*\sin\theta)e^{-iv\xi}d\theta dvdv_*.$$
By using the substitution rule with the new variables $x=v\cos\theta - v_*\sin\theta,\ y=v\sin\theta +v_*\cos\theta,$ we obtain that
\begin{align*}
\mathcal I_{\varepsilon}&=  \iiint_{\{\eps \leq |\theta| \leq \frac{\pi}{4}\}\times\RR^2}\beta(\theta)f(x)\breve g(y) e^{-i(x \cos \theta+y\sin \theta)\xi}d\theta dxdy\\
&=  \ \int_{\eps \leq |\theta| \leq \frac{\pi}{4}}\beta(\theta)\widehat{f}(\xi \cos \theta)\widehat{\breve g}(\xi \sin \theta)d\theta.
\end{align*}
Since
$$\iiint_{\{\eps \leq |\theta| \leq \frac{\pi}{4}\}\times\RR^2}\beta(\theta)\breve g(v_*) f(v)e^{-iv\xi}d\theta dvdv_*=\hat g(0)\hat f(\xi)\int_{\{\eps \leq |\theta| \leq \frac{\pi}{4}\}}\beta(\theta) d\theta,$$
we get that 
\begin{multline*}
\iint_{\{\varepsilon\le \val \theta\le \frac{\pi}{4}\}\times\R}\hspace{-0.5cm}\beta(\theta)\phi_{f,\breve g}(\theta, v)e^{-iv\xi} d\theta dv=
 \int_{\{\eps \leq |\theta| \leq \frac{\pi}{4}\}}\hspace{-0.5cm}\beta(\theta) \bigl(\widehat{f}(\xi \cos \theta)\widehat{\breve g}(\xi \sin \theta)-\hat g(0)\hat f(\xi) \bigr)d\theta.
 \end{multline*}
Lemma~\ref{prop2} follows from \eqref{6.knb55}, \eqref{6.knbgg} and Lemma~\ref{new003}.
\end{proof}

\subsection{Miscellanea}

\subsubsection{The harmonic oscillator}\label{6.sec.harmo}
 The standard Hermite functions $(\phi_{n})_{n\in \N}$ are defined for $x \in \rr$,
 \begin{multline}
 \phi_{n}(x)=\frac{(-1)^n}{\sqrt{2^n n!\sqrt{\pi}}} e^{\frac{x^2}{2}}\frac{d^n}{dx^n}(e^{-x^2})
 =\frac{1}{\sqrt{2^n n!\sqrt{\pi}}} \Bigl(x-\frac{d}{dx}\Bigr)^n(e^{-\frac{x^2}{2}})=\frac{ a_{+}^n \phi_{0}}{\sqrt{n!}},
\end{multline}
where $a_{+}$ is the creation operator 
$$a_{+}=\frac{1}{\sqrt{2}}\Big(x-\frac{d}{dx}\Big).$$
The family $(\phi_{n})_{n\in \N}$ is an orthonormal basis of $L^2(\R)$.
We set for $n\in \N$, $\alpha=(\alpha_{j})_{1\le j\le d}\in\N^d$, $x\in \R$, $v\in \R^d,$
\begin{align}\label{}
\psi_n(x)&=2^{-1/4}\phi_n(2^{-1/2}x),\quad \psi_{n}=\frac{1}{\sqrt{n!}}\Bigl(\frac{x}2-\frac{d}{dx}\Bigr)^n\psi_{0},
\\
\Psi_{\alpha}(v)&=\prod_{j=1}^d\psi_{\alpha_j}(v_j),\quad \mathcal E_{k}=\text{Span}
\{\Psi_{\alpha}\}_{\alpha\in \N^d,\val \alpha=k},
\end{align}
with $\val \alpha=\alpha_{1}+\dots+\alpha_{d}$. The family $(\Psi_{\alpha})_{\alpha \in \nn^d}$ is an orthonormal basis of $L^2(\R^d)$
composed by the eigenfunctions of the $d$-dimensional harmonic oscillator
\begin{equation}\label{6.harmo}
\mathcal{H}=-\Delta_v+\frac{|v|^2}{4}=\sum_{k\ge 0}\Big(\frac d2+k\Big)\mathbb P_{k},\quad \text{Id}=\sum_{k \ge 0}\mathbb P_{k},
\end{equation}
where $\mathbb P_{k}$ is the orthogonal projection onto $\mathcal E_{k}$ whose dimension is $\binom{k+d-1}{d-1}$. The eigenvalue
$d/2$ is simple in all dimensions and $\mathcal E_{0}$ is generated by the function
$$\Psi_{0}(v)=\frac{1}{(2\pi)^{\frac{d}{4}}}e^{-\frac{\val v^2}{4}}=\mu^{1/2}(v),$$
where $\mu$ is the Maxwellian distribution defined in (\ref{maxwe}).

\subsubsection{An asymptotic equivalent}
We consider the integral
\begin{equation}\label{eig00}
\lambda'_{k}=\int_{\val \theta\le \frac{\pi}{4}}\beta(\theta)\bigl(1-(\cos \theta)^k\bigr)d\theta,\quad k\in \N,
\end{equation}
where $\beta$ is the function defined in \eqref{new001'}. We want to prove the following asymptotic equivalent for the integral $\lambda'_{k}$ when $k \to +\io$,
\begin{equation}\label{6.asyvp}
\lambda'_{k}\sim c_{0} k^s\quad\text{with\quad} c_{0}=2^{1+s}\int_{0}^{+\io}{(1-e^{-w})}
\frac{dw}{w^{s+1}}=\frac{2^{1+s}}{s}\Gamma(1-s).
\end{equation}
To that end, we use the substitution rule with $v=2\sin^2(\frac{\theta}{2})$ to obtain that
$$\lambda'_{k}=2^{1+s}\int_{0}^{1-2^{-1/2}}\bigl(1-(1-v)^k\bigr)\frac{dv}{v^{1+s}}=2^{1+s} k^{s}\int_{0}^{k(1-2^{-1/2})}\Bigl(1-\Big(1-\frac w k\Big)^k\Bigr) \frac{dw}{w^{1+s}}.
$$
Furthermore, we have for any $w\in (0,k)$ with $k\ge 1$,
\begin{multline*}
0\le \Bigl(1-\Big(1-\frac w k\Big)^k\Bigr) \frac{1}{w^{1+s}}\le  k \frac{w}{k}\frac{1}{w^{1+s}}\un_{[0,1]}(w)
+\un_{(1,+\io)}(w)\frac{1}{w^{1+s}}\\=\frac{1}{w^{s}}\un_{[0,1]}(w)+\un_{(1,+\io)}(w)\frac{1}{w^{1+s}}\in L^1(\R).
\end{multline*}
It follows from the Lebesgue dominated convergence theorem that
$$\lim_{k\rightarrow+\io} \frac{\lambda'_{k}}{k^{s}}=2^{1+s}\int_{0}^{+\io}\bigl(1-e^{-w}\bigr)\frac{dw}{ w^{1+s}}.$$
We shall now estimate from above the term
\begin{equation}\label{eig01}
\lambda''_{l}=\int_{\val \theta\le \frac{\pi}{4}}\beta(\theta)(\sin\theta)^{2l}
d\theta,\quad l\ge 1.
\end{equation}
We have
\begin{align*}
0\le \lambda''_{l}= & \ 2^{2+2s}\int_{0}^{\frac{\pi}{4}}\frac{\theta^{1+2s}\cos(\frac{\theta}{2})}{2^{1+2s}\sin^{1+2s}(\frac{\theta}{2})}\Big(\frac{\sin\theta}{\theta}\Big)^{2l}\theta^{2l-1-2s}d\theta\\
\le & \  2^{2+2s}\Big(\frac{\pi}{2}\Big)^{1+2s}\int_{0}^{\frac{\pi}{4}}\theta^{2l-1-2s}d\theta= \Big(\frac{\pi}{4}\Big)^{2l-2s}\frac{\pi^{1+2s}}{l-s}\\
\le & \ \frac{4^{2s}\pi }{1-s}\exp{-2l\Big(\log{\frac4\pi}\Big)},
\end{align*}
so that $\lambda''_{l}$ is exponentially decreasing when $l \to +\infty$,
\begin{equation}\label{eig02}
0\le \lambda''_{l}\le \frac{4^{2s}\pi }{1-s}
\exp{-2l\Big(\log{\frac4\pi}\Big)},\quad l\ge 1,\ 0<s<1.
\end{equation}
These estimates prove \eqref{new006}.

\subsubsection{On the Weyl quantization}\label{6.sec.susub}
Let $a$ be a tempered distribution on $\R^d_{v}\times \R^d_{\xi}$. The symbol $a$ may be Weyl quantized in an operator $a^w$ sending $\mathscr S (\R^d)$ into $\mathscr S'(\R^d)$. The formula  \eqref{2.weylq}
is not readily meaningful, but a weak formulation is provided as follows.
We consider the Wigner function of two functions $f,g \in \mathscr S(\R^d)$,
\begin{equation}\label{6.wigner}
(\mathscr W(f,g))(v,\xi)=\frac{1}{(2\pi)^{d}}\int_{\R^{d}} f\Big(v+\frac y2\Big)\overline{g\Big(v-\frac y2\Big)}e^{-iy\cdot\xi}dy.
\end{equation}
We easily check that $ \mathscr W(f,g)$ belongs to $\mathscr S(\R^{2d})$. For $a\in \mathscr S'(\R^{2d})$, we define
\begin{equation}\label{4.kn221}
\poscal{a^w f}{g}_{\mathscr S'(\R^d),\mathscr S(\R^d)}=\poscal{a}{ \mathscr W(f,g)}
_{\mathscr S'(\R^{2d}),\mathscr S(\R^{2d})}.
\end{equation}
The standard formula \eqref{2.weylq} follows from this weak formulation. A nice feature of the Weyl quantization is the fact that
$$(a^w)^*=(\bar a)^w,$$
where $\bar a$ stands for the complex conjugate symbol of $a$. Real-valued symbols are therefore Weyl quantized as formally selfadjoint operators. The distribu\-tion-kernel of the operator $a^w$ is given by
$$k(v,v')=\frac{1}{(2\pi)^{d}}\int_{\rr^d} a\Big(\frac{v+v'}{2},\xi\Big)e^{i(v-v')\cdot \xi}d\xi,$$
where the integral is understood as a partial Fourier transform. Conversely, we deduce from the previous formula that
\begin{equation}\label{4.lk555}
a(v,\xi)=\int_{\rr^d} k\Big(v-\frac{y}{2},v+\frac{y}{2}\Big)e^{i y \cdot \xi}dy,
\end{equation}
where the integral is understood as a partial inverse Fourier transform of the distribution kernel.
A computation in the proofs above has to deal  with the relationship between the  distribution kernel $k(v,y)\in \mathscr S'(\R^d\times \R^d)$ of an operator $A$ and the distribution kernel of the operator $\tilde A$ given by
$$(\tilde A u)(v)= (A\breve u)(v),$$
where $\breve u$ stands for the even part of $u$. An easy computation shows that
\begin{equation}\label{6.kern0}
\tilde k(v,y)=\frac12\bigl(k(v,y)+k(v,-y)\bigr),
\end{equation}
where $\tilde k$ stands for the kernel of $\tilde A$. The formula \eqref{4.lk555} implies that the Weyl symbol of the operator $\tilde A$ is
\begin{align*}
\tilde a(v,\xi)= & \ \frac{1}{2}\int e^{iy \cdot \xi}\Big[k\Bigl(v-\frac y2,v+\frac y2\Bigr)+ k\Bigl(v-\frac y2,-v-\frac y2\Bigr)\Big]dy\\
= & \ \int e^{iy \cdot \xi} \breve k^{\{2\}}\Bigl(v-\frac y2,v+\frac y2\Bigr) dy,
 \end{align*}
 where $ \breve k^{\{2\}}$ stands for the even part of the function $k$ with respect to its second variable.

\subsubsection{On radial functions}\label{6.sec.radia}
If $u\in \mathscr S(\R^d)$ is a radial function
$$\forall x\in\rr^d,\,\,\forall A\in O(d),\  u(x)=u(Ax),$$ 
we define 
$$\tilde u(t)=\frac{1}{\val{\SSS^{d-1}}}\int_{\SSS^{d-1}} u(t\sigma) d\sigma, \quad t\in \R.$$
This function is even, belongs to the Schwartz space $\mathscr S(\R)$ and satisfies 
$$\forall t\in \R, \forall \sigma\in \SSS^{d-1}, \quad  \tilde{u}(t)=u(t\sigma), \quad \forall x\in \R^d, \quad  u(x)=\tilde{u}(\val x).$$
Borel's theorem shows that the mapping $t\mapsto \tilde u(t)$ is also a Schwartz function of the variable $t^2$. We also recall that the Fourier transform of a radial function is radial and that the Fourier transformation is an isomorphism of the space $\mathscr S_{r}(\R^d)$.

\vs\noindent
{\bf Acknowledgements.}
The research of the second author was supported by the Grant-in-Aid for Scientific Research No.22540187, Japan Society of the Promotion of Science. The research of the third author was supported by the chair of excellence CNRS of the Universit\'e de Cergy-Pontoise. The research of the last   author was supported partially by ``The Fundamental Research Funds for Central Universities''. The authors are grateful to Kyoto University and Wuhan University for their kind hospitality and support.

\end{document}